\documentclass[12pt,oneside,english]{amsart}
\usepackage[LGR,T1]{fontenc}
\usepackage[latin9]{inputenc}
\usepackage{geometry}
\usepackage{amstext}
\usepackage{amsthm}
\usepackage{amssymb}
\usepackage[all]{xy}

\makeatletter

\DeclareRobustCommand{\greektext}{%
  \fontencoding{LGR}\selectfont\def\encodingdefault{LGR}}
\DeclareRobustCommand{\textgreek}[1]{\leavevmode{\greektext #1}}
\ProvideTextCommand{\~}{LGR}[1]{\char126#1}

\numberwithin{equation}{section}
\numberwithin{figure}{section}
\theoremstyle{plain}
\newtheorem*{conjecture*}{\protect\conjecturename}
\theoremstyle{plain}
\newtheorem*{thm*}{\protect\theoremname}
\theoremstyle{plain}
\newtheorem{thm}{\protect\theoremname}[section]
\theoremstyle{plain}
\newtheorem{lem}[thm]{\protect\lemmaname}
\theoremstyle{plain}
\newtheorem{prop}[thm]{\protect\propositionname}
\theoremstyle{plain}
\newtheorem{cor}[thm]{\protect\corollaryname}
\theoremstyle{remark}
\newtheorem{rem}[thm]{\protect\remarkname}

\usepackage{amsthm}
\usepackage{amscd}
\usepackage{amsrefs} 
\makeatletter
\@namedef{subjclassname@2020}{2020 Mathematics Subject Classification}
\makeatother

\makeatother

\usepackage{babel}
\providecommand{\conjecturename}{Conjecture}
\providecommand{\corollaryname}{Corollary}
\providecommand{\lemmaname}{Lemma}
\providecommand{\propositionname}{Proposition}
\providecommand{\remarkname}{Remark}
\providecommand{\theoremname}{Theorem}

\begin{document}
\global\long\def\bbR{\mathbb{\mathbf{R}}}%

\global\long\def\R{\mathbb{\mathrm{R}}}%

\global\long\def\C{\mathbb{\mathbf{C}}}%

\global\long\def\Q{\mathbb{\mathbf{Q}}}%

\global\long\def\Z{\mathbf{Z}}%

\global\long\def\P{\mathbb{P}}%

\global\long\def\T{\mathbb{T}}%

\global\long\def\F{\mathbb{\mathbf{F}}}%

\global\long\def\bbN{\mathbb{\mathbf{N}}}%

\global\long\def\M{\mathrm{M}}%

\global\long\def\A{\mathrm{{\bf A}}}%

\global\long\def\H{\mathrm{H}}%

\global\long\def\D{\mathrm{\mathbf{D}}}%

\global\long\def\B{\mathbf{B}}%

\global\long\def\N{\mathrm{\mathrm{N}}}%

\global\long\def\inf{\mathrm{inf}}%

\global\long\def\sup{\mathrm{sup}}%

\global\long\def\Hom{\mathbb{\mathrm{Hom}}}%

\global\long\def\Ext{\mathbb{\mathbb{\mathrm{Ext}}}}%

\global\long\def\Ker{\mathbb{\mathrm{Ker}}}%

\global\long\def\Gal{\mathrm{Gal}}%

\global\long\def\Aut{\mathrm{Aut}}%

\global\long\def\End{\mathrm{End}}%

\global\long\def\Char{\mathrm{Char}}%

\global\long\def\rank{\mathrm{rank}}%

\global\long\def\deg{\mathrm{deg}}%

\global\long\def\det{\mathrm{det}}%

\global\long\def\Tr{\mathrm{Tr}}%

\global\long\def\Id{\mathrm{Id}}%

\global\long\def\Spec{\mathrm{Spec}}%

\global\long\def\Lie{\mathrm{Lie}}%

\global\long\def\span{\mathrm{span}}%

\global\long\def\sep{\mathrm{sep}}%

\global\long\def\sgn{\mathrm{sgn}}%

\global\long\def\dist{\mathrm{dist}}%

\global\long\def\val{\mathrm{val}}%

\global\long\def\dR{\mathrm{dR}}%

\global\long\def\st{\mathrm{st}}%

\global\long\def\rig{\mathrm{rig}}%

\global\long\def\cris{\mathrm{cris}}%

\global\long\def\Mat{\mathrm{Mat}}%

\global\long\def\cyc{\mathrm{cyc}}%

\global\long\def\et{\mathrm{\acute{e}t}}%

\global\long\def\Frob{\mathrm{Frob}}%

\global\long\def\SL{\mathrm{SL}}%

\global\long\def\GL{\mathrm{GL}}%

\global\long\def\Spa{\mathrm{\mathrm{Spa}}}%

\global\long\def\Br{\mathrm{Br}}%

\global\long\def\Ind{\mathrm{Ind}}%

\global\long\def\LT{\mathrm{LT}}%

\global\long\def\Res{\mathrm{Res}}%

\global\long\def\SL{\mathrm{SL}_{2}}%

\global\long\def\Mod{\mathrm{Mod}}%

\global\long\def\sm{\mathrm{sm}}%

\global\long\def\lsm{\mathrm{lsm}}%

\global\long\def\psm{\mathrm{psm}}%

\global\long\def\plsm{\mathrm{plsm}}%

\global\long\def\Fil{\mathrm{Fil}}%

\global\long\def\la{\mathrm{la}}%

\global\long\def\pa{\mathrm{pa}}%

\global\long\def\Sen{\mathrm{Sen}}%

\global\long\def\ac{\mathrm{ac}}%

\global\long\def\dif{\mathrm{dif}}%

\global\long\def\HT{\mathrm{HT}}%

\global\long\def\FT{\mathrm{FT}}%

\global\long\def\FF{\mathrm{FF}}%

\global\long\def\Kla{K-\mathrm{la}}%

\global\long\def\Kpa{K-\mathrm{pa}}%

\global\long\def\an{\mathrm{an}}%

\global\long\def\han{\text{-an}}%

\global\long\def\hla{\text{-la}}%

\global\long\def\hpa{\text{-pa}}%

\global\long\def\halg{\text{-alg}}%

\global\long\def\Sol{\mathrm{Sol}}%

\global\long\def\grad{\mathrm{\triangledown}}%

\global\long\def\WL{\mathrm{\widetilde{\Lambda}}}%

\global\long\def\weak{\rightharpoonup}%

\global\long\def\weakstar{\overset{*}{\rightharpoonup}}%

\global\long\def\diam{\diamond}%

\title{Overconvergence of étale $(\varphi,\Gamma)$-modules in families}
\author{Gal Porat}
\begin{abstract}
We prove a conjecture of Emerton, Gee and Hellmann concerning the
overconvergence of étale $(\varphi,\Gamma)$-modules in families parametrized
by topologically finite type $\Z_{p}$-algebras. As a consequence,
we deduce the existence of a natural map from the rigid fiber of the
Emerton-Gee stack to the rigid analytic stack of $(\varphi,\Gamma)$-modules.
\end{abstract}

\email{galporat1@gmail.com}

\maketitle
\tableofcontents{}

\section{Introduction}

In recent years, there has been growing interest in realizing the
collection of Langlands parameters in various settings as a moduli
space with a geometric structure. In particular, in the $p$-adic
Langlands program for $\Gal(\overline{K}/K)$, where $K$ is a finite
extension of $\Q_{p}$, this space should come in two different forms
corresponding to the two different flavours of the $p$-adic Langlands
correspondence. In the so called ``Banach'' case, this space has
been constructed and studied by Emerton and Gee in their manuscript
\cite{EG19}. For $d\geq1$, it is the moduli stack $\mathcal{X}_{d}$
whose value on a $\Z/p^{a}$-finite type scheme $\mathrm{Spec}A$
is the groupoid of $d$-dimensional projective étale $(\varphi,\Gamma)$-modules
over the ring $\A_{K,A}$. In the ``analytic'' case, this space
has been defined by Emerton, Gee and Hellmann in \cite{EGH22}. For
$d\geq1$, it is the rigid analytic stack $\mathfrak{X}_{d}$ whose
value on an affinoid space $\mathrm{Sp}A$ is the groupoid of $d$-dimensional
projective $(\varphi,\Gamma)$-modules over the Robba ring $\mathcal{R}_{K,A}$.

These two cases are believed to be related. In \cite{EGH22}, the
authors express the following expectation: there should exist a map
\[
\pi_{d}:\mathcal{X}_{d}^{\rig}\rightarrow\mathfrak{X}_{d}
\]
obtained by forgetting the étale lattice. Unfortunately, the existence
of $\pi_{d}$ is not immediate from the definitions. Indeed, if $A$
is a $p$-adically complete, topologically of finite type $\Z_{p}$-algebra,
there is no natural map from $\A_{K,A}$ to $\mathcal{R}_{K,A}$.
Rather, there is a ring $\A_{K,A}^{\dagger}$ of overconvergent periods
which is naturally contained in both of $\A_{K,A}$ and $\mathcal{R}_{K,A}$,
and it is through this subring that we obtain the link between the
two types of $(\varphi,\Gamma)$-modules. In light of this, Emerton,
Gee and Hellmann make the following conjecture \cite{EGH22}.
\begin{conjecture*}
Every étale $(\varphi,\Gamma)$-module over $\A_{K,A}$ canonically
descends to an étale $(\varphi,\Gamma)$-module over $\A_{K,A}^{\dagger}$.
Consequently, the map $\pi_{d}$ exists.
\end{conjecture*}
The main result of this article confirms this expectation.
\begin{thm*}
The conjecture is true.

More precisely, the functor $M^{\dagger}\mapsto M:=\A_{K,A}\otimes_{\A_{K,A}^{\dagger}}M^{\dagger}$
induces an equivalence of categories from the category of projective
étale $(\varphi,\Gamma)$-modules over $\A_{K,A}^{\dagger}$ to the
category of projective étale $(\varphi,\Gamma)$-modules over $\A_{K,A}$.
\end{thm*}

\subsection{Previous results}

The main result of this article is already known in some cases. The
first result of overconvergence, proved by Cherbonnier\textendash Colmez
in \cite{CC98}, can be thought of as the case $A=\Z_{p}$ of the
conjecture. Later, Berger and Colmez in \cite{BC08} extended these
ideas. Though the current setting is a little different, loc. cit.
makes it clear that overconvergence of free étale $(\varphi,\Gamma)$-modules
would hold whenever the family comes from a family of Galois representations,
at least in the case $A$ is $p$-torsionfree. Two other works worth
mentioning are those of Gao (\cite{Ga19}), which establishes overconvergence
in the case $A=\Z_{p}$ without appealing to Galois representations;
and Bellovin's article (\cite{Bel20}), which proves overconvergence
of families of étale $(\varphi,\Gamma)$-modules coming from Galois
representations in the pseudorigid setting. One novel feature of the
latter two works is that they prove overconvergence in settings where
$A$ could have $p$-torsion.

\subsection{The ideas of the proof}

The difficult part of the theorem is the essential surjectivity of
the functor. The scheme of the proof is to introduce two additional
perfect rings of periods $\widetilde{\A}_{K,A},\widetilde{\A}_{K,A}^{\dagger}$,
with inclusions
\[
\A_{K,A}\subset\widetilde{\A}_{K,A}\supset\widetilde{\A}_{K,A}^{\dagger}\supset\A{}_{K,A}^{\dagger}.
\]
Then, starting with a projective étale $(\varphi,\Gamma)$-module
over $\A_{K,A}$, we extend it to $\widetilde{\A}_{K,A}$, and then
descend in two steps, first from $\widetilde{\A}_{K,A}$ to $\widetilde{\A}_{K,A}^{\dagger}$
and then from $\widetilde{\A}_{K,A}^{\dagger}$ to $\A{}_{K,A}^{\dagger}$.

Let us emphasize that when a family of étale $(\varphi,\Gamma)$-modules
comes from a family of Galois representations, the first descent step
from $\widetilde{\A}_{K,A}$ to $\widetilde{\A}_{K,A}^{\dagger}$
can be completely avoided. Therefore in previous work, the entire
focus was on the second step. However, outside of the case where $A$
is a finite type $\Z_{p}$-algebra, families of étale $(\varphi,\Gamma)$-modules
over $\A_{K,A}$ do not in general come from Galois representations
(\cite{Ch09,KL11}), so we want to avoid using them. 

Thus, a method for descending from $\widetilde{\A}_{K,A}$ to $\widetilde{\A}_{K,A}^{\dagger}$
is required. Fortunately, this was worked out in the case $A=\Z_{p}$
in the article of de Shalit and the author in \cite{dSP19}, based
on the author's master thesis, by using the contracting properties
of Frobenius. The first key idea of this article is to generalize
these results to the setting of families; since the method of \cite{dSP19}
is relatively elementary, this does not cause too many technical difficulties.

Next, for the descent step from $\widetilde{\A}_{K,A}^{\dagger}$
to $\A{}_{K,A}^{\dagger}$, we use ideas based on the Tate\textendash Sen
method of \cite{BC08}. However, to apply the descent results of loc.
cit., one needs the existence of an open subgroup for which the matrices
of the group action are congruent to 1 mod some power of $p$, which
is arranged there by restricting attention to étale $(\varphi,\Gamma)$-modules
coming from Galois representations, and using the lattice of such
a representation. This causes two technical problems for us: first,
we want to avoid using Galois representations; and second, we work
in a setting where one cannot expect such a congruence to occur in
general. The second key idea of this article is to develop a variant
of the Tate\textendash Sen method for Tate rings which replaces the
role of $p$ by a pseudouniformizer $f$. Here, we were inspired by
the article \cite{Bel20} which showed how the role of $p$ can be
replaced by that of a psuedouniformizer (though still in the context
of Galois representations). This idea ends up solving both of the
aforementioned problems at the same time. In fact, in terms of applications,
our method turns out to be flexible enough to also reprove results
of both \cite{BC08} and \cite{Bel20}.

\subsection{Structure of the article}

In $\mathsection2$, we give the definitions and recall the basic
set up. In $\mathsection3$ we prove the descent from $\widetilde{\A}_{K,A}$
to $\widetilde{\A}_{K,A}^{\dagger}$. In $\mathsection4$, we develop
the variant of the Tate\textendash Sen method to be used in $\mathsection5$.
This might be of independent interest for future applications. In
$\mathsection5$, we prove the descent from $\widetilde{\A}_{K,A}^{\dagger}$
to $\A_{K,A}^{\dagger}$. Finally, $\mathsection6$ is a short section
putting everything together for the proof of the main theorem. In
$\mathsection7$ we add an appendix, verifying properties of the coefficients
rings we use that are not covered by existing literature.

\subsection{Notations and conventions}

The field $K$ denotes a finite extension of $\Q_{p}$. The group
$G_{K}=\Gal(\overline{K}/K)$ denotes the absolute Galois group of
$K$. We write $K_{\infty}=K(\mu_{p^{\infty}})$ for the cyclotomic
extension. Its absolute Galois group is $H_{K}=\Gal(\overline{K}/K_{\infty})$,
and the cyclotomic character identifies the quotient $\Gamma_{K}=G_{K}/H_{K}\cong\Gal(K_{\infty}/K)$
with an open subgroup of $\Z_{p}^{\times}$.

We write $\C$ for the $p$-adic completion of an algebraic closure
of $\Q_{p}$, and $\widehat{K}_{\infty}$ for the $p$-adic completion
of $K_{\infty}$. Both of these are perfectoid fields. We choose a
compatible system of roots of unity $\zeta_{p},\zeta_{p^{2}},...$
and let $\varepsilon=(\zeta_{p},\zeta_{p^{2}},...)$. We let $\varpi=\varepsilon-1$;
it is a pseudouniformizer of both $\widehat{K}_{\infty}^{\flat}$
and $\C^{\flat}$, and has valuation $p/p-1$.

By a valuation on a ring $R$, we mean a map $\val_{R}:R\rightarrow(-\infty,\infty]$
satisfying the following properties for $x,y\in R$:
\begin{enumerate}
\item $\val_{R}(x)=\infty$ if and only if $x=0$ (i.e. $R$ is separated);
\item $\val_{R}(xy)\geq\val_{R}(x)+\val_{R}(y)$;
\item $\val_{R}(x+y)\geq\min(\val_{R}(x),\val_{R}(y))$.
\end{enumerate}
Occasionally, we also allow $\val_{R}(x)=-\infty$; we shall point
out when this is the case. Given a matrix $M$ with coefficients in
$R$, we let $\val_{R}(M)$ be the minimum of the valuation of its
entries.

Whenever we introduce a module or a ring as an inverse limit (resp.
direct limit or localization) of topological modules or rings, we
always endow it with the inverse limit topology (resp. direct limit
topology\footnote{Recall that the direct limit topology on a direct limit $X=\varinjlim X_{i}$
of topological spaces is the finest topology on $X$ for which each
$X_{i}\rightarrow X$ is continuous.}) unless otherwise stated.

If $R$ is a topological ring endowed with continuous, commuting $(\varphi,\Gamma_{K})$-actions,
then:
\begin{itemize}
\item an étale $\varphi$-module $M$ over $R$ is a finitely generated
$R$-module together with a $\varphi$-semilinear continuous map $\varphi:M\rightarrow M$,
such the linearized morphism $\varphi^{*}M\rightarrow M$ is an isomorphism.
\item an étale $(\varphi,\Gamma_{K})$-module $M$ over $R$ is an étale
$\varphi$-module together with a continuous semilinear action of
$\Gamma_{K}$ such that the actions of $\varphi$ and $\Gamma_{K}$
commute.
\end{itemize}

\subsection{Acknowledgments}

I wish to thank Rebecca Bellovin, Matthew Emerton, Toby Gee and Eugen
Hellmann for helpful discussions. Thanks also to Rebecca Bellovin
and Toby Gee for their helpful comments on the first draft of this
article. Next, I want to thank Lie Qian for informing me of the paper
\cite{Bel20} during the Arizona winter school of 2022. Special thanks
goes to Ehud de Shalit, who advised me when I was a master student,
for teaching me tricks you can do with the Frobenius map. The conjecture
regarding the existence of a natural map from $\mathcal{X}_{d}^{\rig}$
to $\mathfrak{X}_{d}$ was presented in a course given by Matthew
Emerton, Toby Gee and Eugen Hellmann in the IHES summer school on
the Langlands program, in which the author was present. It is during
the school that the key ideas of the proof were understood. I would
like to thank the organizers of the school for the pleasant event,
and the institution for its hospitality.

\section{Basic set up}

In this section we introduce the rings of coefficients which will
play a role in this article and recall some of their properties. These
rings are a little bit different than most of these appearing in most
of the literature on overconvergence of $(\varphi,\Gamma)$-modules,
so details about them do not seem to exist in this generality (outside
of \cite{EG19}). To keep the article readible, we have moved all
proofs of this claims made in this section to an appendix.

\subsection{The rings }

We introduce the following objects.
\begin{itemize}
\item Let $A$ be a $p$-adically complete $\Z_{p}$-algebra which is topologically
of finite type. 
\item Set $\widetilde{{\bf A}}^{+}=\A_{\inf}:=W(\mathcal{O}_{\C}^{\flat})$
and $\widetilde{{\bf \A}}:=W(\C^{\flat})$. We endow $\widetilde{\A}^{+}$
with its $(p,[\varpi])$-topology. We endow $\widetilde{{\bf \A}}=\varprojlim_{i}\widetilde{{\bf \A}}/p^{i}$
with the inverse limit topology, where the topology on $\widetilde{{\bf \A}}/p^{i}=(\widetilde{{\bf \A}}^{+}/p^{i})[1/[\varpi]]$
is characterized by having $\widetilde{{\bf A}}^{+}/p^{i}$ as an
open subring.
\item For $1/r\in\Z[1/p]_{>0}$ we write
\[
\widetilde{\A}^{(0,r],\circ}:=\left\{ \sum_{k\geq0}p^{k}[x_{k}]\in\widetilde{\A}:0\leq\val_{\C^{\flat}}(x_{k})+(\frac{pr}{p-1})k\rightarrow\infty\right\} .
\]
According to \cite[Rem. II.1.3]{CC98}, this ring may be equivalently
defined as $\widetilde{\A}^{+}\left\langle p/[\varpi]^{1/r}\right\rangle $,
the $p$-adic completion of $\widetilde{\A}^{+}[p/[\varpi]^{1/r}]$.
We endow the ring $\widetilde{\A}^{(0,r],\circ}$ with its $[\varpi]$-adic
topology, or what is the same, the topology defined by the valuation
\[
\val^{(0,r]}(x):=(\frac{p}{p-1})\sup\left\{ t\in\Z[1/p]:x\in[\varpi]^{t}\widetilde{\A}^{(0,r],\circ}\right\} .
\]
One checks that 
\[
\val^{(0,r]}(\sum_{k\geq0}p^{k}[x_{k}])=\inf_{k}[\val_{\C^{\flat}}(x_{k})+(\frac{pr}{p-1})k].
\]
\item For $1/r\in\Z[1/p]_{>0}$ we set
\[
\widetilde{\A}^{(0,r]}:=\left\{ x=\sum_{k\gg-\infty}p^{k}[x_{k}]\in\widetilde{\A}:\val_{\C^{\flat}}(x_{k})+(\frac{pr}{p-1})k\rightarrow\infty\right\} .
\]
Equivalently, $\widetilde{\A}^{(0,r]}=\widetilde{\A}^{(0,r],\circ}\left[1/[\varpi]\right]$.
We endow $\widetilde{\A}^{(0,r]}$ with the topology which makes $\widetilde{\A}^{(0,r],\circ}$
into an open subring. This topology is the same as that defined by
the valuation 
\[
\val^{(0,r]}(x):=(\frac{p}{p-1})\sup\left\{ t\in\Z[1/p]:x\in[\varpi]^{t}\widetilde{\A}^{(0,r],\circ}\right\} .
\]
\item For $r=\infty$ we set $\widetilde{\A}^{(0,\infty),\circ}=\widetilde{\A}^{+}$
with its $(p,[\varpi])$-topology and $\widetilde{\A}^{(0,\infty)}=\widetilde{\A}^{+}[1/[\varpi]]$
with its $p$-adic topology. 
\end{itemize}
For each of the rings introduced above, there are versions with coefficients
in $A$, which we introduce next. We have:
\begin{itemize}
\item For $1/r\in\Z[1/p]_{>0}$, 
\[
\widetilde{\A}_{A}^{(0,r],\circ}=\widetilde{\A}^{(0,r],\circ}\widehat{\otimes}_{\Z_{p}}A:=\varprojlim_{i}(\widetilde{\A}^{(0,r],\circ}\otimes_{\Z_{p}}A)/[\varpi^{i}]
\]
 and $\widetilde{\A}_{A}^{(0,r]}=\widetilde{\A}_{A}^{(0,r],\circ}\left[1/[\varpi]\right]$
\item $\widetilde{\A}_{A}^{(0,r],+}:=$ the image of $\widetilde{\A}_{A}^{(0,r],\circ}$
in $\widetilde{\A}_{A}^{(0,r]}$, endowed with its subspace topology.
(The map $\widetilde{\A}_{A}^{(0,r],\circ}\rightarrow\widetilde{\A}_{A}^{(0,r]}$
may not be injective if $A$ has $p$-torsion). 
\item For $a\geq1$ the $a$-typical Witt vectors $W_{a}(\mathcal{O}_{\C}^{\flat})=W(\mathcal{O}_{\C}^{\flat})/p^{a}$,
with the $[\varpi]$-adic topology.
\item For $a\geq1$,  
\[
W_{a}(\mathcal{O}_{\C}^{\flat})_{A}=W_{a}(\mathcal{O}_{\C}^{\flat})\widehat{\otimes}_{\Z_{p}}A:=\varprojlim_{i}(W_{a}(\mathcal{O}_{\C}^{\flat})\otimes_{\Z_{p}}A)/[\varpi^{i}].
\]
.
\item For $r=\infty$, $\widetilde{\A}_{A}^{+}=\widetilde{\A}_{A}^{(0,\infty),+}=\widetilde{\A}_{A}^{(0,\infty),\circ}:=W(\mathcal{O}_{\C}^{\flat})_{A}=\varprojlim_{a}W_{a}(\mathcal{O}_{\C}^{\flat})_{A}$
 and $\widetilde{\A}_{A}^{(0,\infty)}=W(\mathcal{O}_{\C}^{\flat})_{A}\left[1/[\varpi]\right]$.
\item $\widetilde{\A}_{A}=W(\C^{\flat})_{A}:=\varprojlim_{a}W_{a}(\mathcal{O}_{\C}^{\flat})_{A}[1/[\varpi]]$.
\item $\widetilde{\A}_{A}^{\dagger}:=\varinjlim_{r}\widetilde{\A}_{A}^{(0,r]}$.
\end{itemize}
For each $R_{A}\in\{\widetilde{\A}_{A}^{+},\widetilde{\A}_{A},\widetilde{\A}_{A}^{(0,r]},\widetilde{\A}_{A}^{(0,\infty]},\widetilde{\A}_{A}^{\dagger}\}$
we introduce versions relative to $K$, by setting $R_{K,A}:=(R_{A})^{H_{K}}$.
The following lemma shows that each of these can be defined in terms
of $\widehat{K}_{\infty}$. Define $W(\mathcal{O}_{\widehat{K}_{\infty}}^{\flat})_{A},W(\widehat{K}_{\infty}^{\flat})_{A}$
in a similar way to the definition of $W(\mathcal{O}_{\C}^{\flat})_{A}$,
$W(\C^{\flat})_{A}$ respectively.
\begin{lem}
(Lemma 7.1) We have natural isomorphisms

i. $\widetilde{\A}_{K,A}^{+}\cong W(\mathcal{O}_{\widehat{K}_{\infty}}^{\flat})_{A}.$

ii. $\widetilde{\A}_{K,A}\cong W(\widehat{K}_{\infty}^{\flat})_{A}$.

iii. $\widetilde{\A}_{K,A}^{(0,r]}\cong(\widetilde{\A}_{K}^{+}\left\langle p/[\varpi]^{1/r}\right\rangle \widehat{\otimes}A)[1/[\varpi]]$.

iv. $\widetilde{\A}_{K,A}^{(0,\infty)}\cong W(\mathcal{O}_{\widehat{K}_{\infty}}^{\flat})_{A}\left[1/[\varpi]\right]$.

v. $\widetilde{\A}_{K,A}^{\dagger}=\varinjlim_{r}(\widetilde{\A}_{K}^{+}\left\langle p/[\varpi]^{1/r}\right\rangle \widehat{\otimes}A)[1/[\varpi]]$.

Here, in $iii$ and $v$, the tensor products are completed with respect
to the $[\varpi]$-adic topology.
\end{lem}

Finally, we have imperfect versions of the rings relative to $K$
as above. They are defined as follows.
\begin{itemize}
\item We have standard rings $\A_{K}^{+}$ and $\A_{K}$, which are defined
in $\mathsection2.1$ of \cite{EG19} (where they are denoted by $(\A_{K}^{\prime})^{+}$
and $\mathbf{\A}_{K}^{\prime}$). We have a certain element $T\in\A_{K}^{+}$
lifting $\varpi$. Note that by the theory of the field of norms,
we have canonical embeddings $\A_{K}^{+}\hookrightarrow\widetilde{\A}$
and $\A_{K}\hookrightarrow\widetilde{\A}$ which map $T$ to $[\varepsilon]-1$.
\item We set\footnote{Compare with $\mathsection2.2$ of loc. cit. Note that there, $T$
can be used in the corresponding definitions instead of $T_{K}$. }
\[
\A_{K,A}^{+}=\A_{K}^{+}\widehat{\otimes}_{\Z_{p}}A:=\varprojlim_{m}(\varprojlim_{n}\A_{K}^{+}/(p^{m},T^{n})\otimes_{\Z_{p}}A),
\]
endowed with the inverse limit topology, and
\[
\A_{K,A}=\A_{K}\widehat{\otimes}_{\Z_{p}}A:=\varprojlim_{m}(\varprojlim_{n}(\A_{K}^{+}/(p^{m},T^{n})\otimes_{\Z_{p}}A)[1/T]).
\]
\item We let $\A_{K}^{(0,r],\circ}=\A_{K}\cap\widetilde{\A}^{(0,r],\circ}$
and 
\[
\A_{K,A}^{(0,r],\circ}=\varprojlim_{i}(\A_{K}^{(0,r],\circ}\otimes_{\Z_{p}}A)/T^{i}.
\]
\item Let $\A_{K,A}^{(0,r]}:=\A_{K,A}^{(0,r],\circ}[1/T]$ and write $\A_{K,A}^{(0,r],+}$
for the image of $\A_{K,A}^{(0,r],\circ}$ in $\A_{K,A}^{(0,r]}$.
\item Finally, let $\A_{K,A}^{\dagger}:=\varinjlim_{r}\A_{K,A}^{(0,r]}$. 
\end{itemize}
We have natural maps $\widetilde{\A}_{A}^{(0,r]}\rightarrow\widetilde{\A}_{A}$,
$\A_{K,A}\rightarrow\widetilde{\A}_{A}$, $\widetilde{\A}_{A}^{(0,s]}\rightarrow\widetilde{\A}_{A}^{(0,r]}$
and $\A_{K,A}^{(0,s]}\rightarrow\A_{K,A}^{(0,r]}$ for $s>r$. One
can show they are all injective (this claim is not trivial - it is
the main content of the appendix, cf Theorem 7.9 and what follows).
This implies that in the definition of $\widetilde{\A}_{A}^{\dagger}$,
the colimit is in fact a union. Similar comments apply to $\widetilde{\A}_{K,A}^{\dagger}$
and $\A_{K,A}^{\dagger}$.

We shall need the following fact on compatibility with reduction.
\begin{prop}
(Proposition 7.17) For $N\in\Z_{\geq1}$ and $1/r\in\Z[1/p]_{>0}$
we have natural isomorphisms 
\[
\widetilde{\A}_{K,A}/p^{N}\widetilde{\A}_{K,A}\cong\widetilde{\A}_{K,A/p^{N}}\cong\widetilde{\A}_{K,A/p^{N}}^{(0,r]}\cong\widetilde{\A}_{K,A/p^{N}}^{(0,\infty)}.
\]
\end{prop}

The following result that will be used later in $\mathsection5$.
\begin{prop}
(Corollary 7.4) For $1/r\in\Z[1/p]_{>0}$ the topology on $\widetilde{\A}_{A}^{(0,r]}$
is defined by the valuation given by
\[
\val^{(0,r]}(x)=(p/p-1)\sup\{t\in\Z[1/p]:x\in[\varpi]^{t}\widetilde{\A}_{A}^{(0,r],+}\}.
\]
\end{prop}

Finally, there are continuous $(\varphi,G_{K})$-actions on $\widetilde{\A}_{A}$
by \cite[Lem. 2.2.18]{EG19}. This gives rise to continuous $(\varphi,\Gamma_{K})$-actions
on all of the rings introduced here.

Namely, given $r\in\bbR_{>0}\cup\{\infty\}$ we have a continuous
map 
\[
\varphi:\widetilde{\A}_{A}^{(0,r]}\rightarrow\widetilde{\A}_{A}^{(0,r/p]}
\]
induced by extending the action of $\varphi$ on $\A_{\inf}$. It
is continuous since $\varphi([\varpi])=[\varpi]^{p}$ and the topology
is $[\varpi]$-adic on both the source and the target.

Similarly, we have a continuous inverse
\[
\varphi^{-1}:\widetilde{\A}_{A}^{(0,r/p]}\rightarrow\widetilde{\A}_{A}^{(0,r]}.
\]

This immediately extends to give continuous $\varphi$ and $\varphi^{-1}$
actions on $\widetilde{\A}_{A}^{\dagger}$. The $(\varphi^{\pm1},G_{K})$-actions
(resp. $(\varphi^{\pm1},\Gamma_{K})$-actions, resp. $(\varphi,\Gamma_{K})$-actions)
on $\widetilde{\A}_{A}^{\dagger}$ and $\widetilde{\A}_{A}$ (resp.
$\widetilde{\A}_{K,A}^{\dagger}$ and $\widetilde{\A}_{K,A}$, resp.
$\A_{K,A}^{\dagger}$ and $\A_{K,A}$) are continuous. 

\section{Overconvergence for perfect coefficients}

The purpose of this section is to prove the following result.
\begin{thm}
The functor $\widetilde{M}^{\dagger}\mapsto\widetilde{M}:=\widetilde{M}^{\dagger}\otimes_{\widetilde{\A}_{K,A}^{\dagger}}\widetilde{\A}_{K,A}$
induces an equivalence of categories from the category of projective
\textup{\emph{é}}tale $(\varphi,\Gamma_{K})$-modules over $\widetilde{\A}_{K,A}^{\dagger}$
to the category of projective \textup{\emph{é}}tale $(\varphi,\Gamma_{K})$-modules
over $\widetilde{\A}_{K,A}$.
\end{thm}

Our proof follows $\mathsection4$ of \cite{dSP19} with appropriate
modifications. The idea will be to first establish the equivalence
for étale $\varphi$-modules and then deduce it for étale $(\varphi,\Gamma_{K})$-modules.

\subsection{Descent of $\varphi$-modules}

The topology on $\widetilde{\A}_{K,A}^{(0,r]}$ is defined by a valuation
$\val^{(0,r]}$
\[
\val_{A}^{(0,r]}(x):=(\frac{p}{p-1})\sup\{t\in\Z[1/p]:x\in[\varpi]^{t}\widetilde{\A}_{K,A}^{(0,r],+}\}.
\]
We also define for $x\in\widetilde{\A}_{K,A}^{(0,\infty)}$
\[
\val_{A}^{(0,\infty)}(x):=(\frac{p}{p-1})\sup\{t\in\Z[1/p]:x\in[\varpi]^{t}\widetilde{\A}_{K,A}^{(0,\infty),+}\}.
\]
We extend $\val_{A}^{(0,r]}$ and $\val_{A}^{(0,\infty)}$ to all
of $\widetilde{\A}_{K,A}$, by allowing the value $-\infty$, so that
$\widetilde{\A}_{K,A}^{(0,r]}$ (resp. $\widetilde{\A}_{K,A}^{(0,\infty)}$)
is the subset of elements $x\in\widetilde{\A}_{K,A}$ with $\val_{A}^{(0,r]}(x)>-\infty$
(resp. $\val_{A}^{(0,\infty)}(x)>-\infty$).

The following lemma serves as a generalization of the usual Teichmüller
digits. Recall from Proposition 2.2 that 
\[
\widetilde{\A}_{K,A}/p^{N}\cong\widetilde{\A}_{K,A/p^{N}}\cong\widetilde{\A}_{K,A/p^{N}}^{(0,\infty)}\cong\widetilde{\A}_{K,A/p^{N}}^{(0,r]}.
\]

\begin{lem}
There exists a (noncanonical) map $[\bullet]:\widetilde{\A}_{K,A}/p\rightarrow\widetilde{\A}_{K,A}$
such that
\end{lem}

1. For every $x\in\widetilde{\A}_{K,A}/p$ we have $[x]\equiv x$
mod $p$.

2. The map $[\bullet]$ commutes with $\varphi$.

3. For every every $x\in\widetilde{\A}_{K,A}/p$, we have $\val_{A}^{(0,\infty)}([x])\geq\val_{A/p}^{(0,\infty)}(x)$.

4. For every $n\in\Z_{\geq0}$ and every $x\in\widetilde{\A}_{K,A}/p$,
we have $\val_{A}^{(0,r]}(p^{n}[x])\geq\val_{A/p^{n+1}A}^{(0,r]}(p^{n}[x]\mod p^{n+1})$.
\begin{proof}
Multiplication by $p$ induces surjections
\[
A/p\twoheadrightarrow pA/p^{2}\twoheadrightarrow p^{2}A/p^{3}\twoheadrightarrow...
\]
of $\F_{p}$-vector spaces. Let $W_{n}$ be the kernel of $\pi_{n}:A/p\twoheadrightarrow p^{n}A/p^{n+1}$,
so that the $W_{n}$ are increasing with $n$. Choose a complement
$U$ to the union of the $W_{n}$ so that
\[
A/p=U\oplus\bigcup_{n\geq0}W_{n}.
\]
For $U$ choose an $\F_{p}$-basis $\{e_{i}\}_{i\in I_{U}}$ and choose,
as we may, compatible bases $\{e_{i}\}_{i\in I_{n}}$ for $W_{n}$
so that $\{e_{i}\}_{i\in I_{n}}\subset\{e_{i}\}_{i\in I_{n+1}}$,
and set $I=\bigcup_{n\geq0}I_{n}$. With this being given, we set
$J_{n}=I_{U}\cup(I\backslash I_{n})$ for $n\geq0$. We claim that
$\{\pi_{n}(e_{i})\}_{i\in J_{n}}$ gives an $\F_{p}$-basis of $p^{n}A/p^{n+1}$.
Indeed, by construction, the map $\pi_{n}:A/p\twoheadrightarrow p^{n}A/p^{n+1}$
gives a decomposition 
\[
A/p=\span\{e_{i}\}_{i\in I_{U}\cup I}=W_{n}\oplus\span\{e_{i}\}_{i\in J_{n}},
\]
so that applying $\pi_{n}$ to the $e_{i}$ for $i\in J_{n}$ gives
an $\F_{p}$-basis of $p^{n}A/p^{n+1}$. With this choice of $\{e_{i}\}_{i\in I\cup I_{U}}$,
choose an arbitrary lifting of them to $A$, which we denote by $\widetilde{e}_{i}$. 

Now each $x\in\mathcal{O}_{\widehat{K}_{\infty}}^{\flat}\otimes_{\F_{p}}A/p$
can be written uniquely as a finite sum of the form $x=\sum x_{i}\otimes e_{i}$
with $x_{i}\in\mathcal{O}_{\widehat{K}_{\infty}}^{\flat}$. Let
\[
[x]:=\sum[x_{i}]\otimes\widetilde{e_{i}}\in W(\mathcal{O}_{\widehat{K}_{\infty}}^{\flat})\otimes_{\Z_{p}}A,
\]

where $[x_{i}]$ is the usual Teichmüller lift. This defines a map
\[
[\bullet]:\mathcal{O}_{\widehat{K}_{\infty}}^{\flat}\otimes_{\F_{p}}A/p\rightarrow W(\mathcal{O}_{\widehat{K}_{\infty}}^{\flat})\otimes_{\Z_{p}}A.
\]
It follows from the definition that 
\[
[\bullet](\varpi^{t}\mathcal{O}_{\widehat{K}_{\infty}}^{\flat}\otimes_{\F_{p}}A/p)\subset[\varpi]^{t}W(\mathcal{O}_{\widehat{K}_{\infty}}^{\flat})\otimes_{\Z_{p}}A,
\]
hence it is continuous for the $[\varpi]$-topology on the source
and the $(p,[\varpi])$-topology on the target. It therefore extends
to a map
\[
[\bullet]:(\mathcal{O}_{\widehat{K}_{\infty}}^{\flat})_{A/p}\rightarrow W(\mathcal{O}_{\widehat{K}_{\infty}}^{\flat})_{A},
\]
which we can extend further to a continuous map
\[
[\bullet]:\widetilde{\A}_{K,A/p}=(\mathcal{O}_{\widehat{K}_{\infty}}^{\flat})_{A/p}[1/\varpi]\rightarrow W(\widehat{K}_{\infty}^{\flat})_{A}=\widetilde{\A}_{K,A}.
\]
Clearly properties 1 and 2 are satisfied. 

To show property 3 holds, we need to check that $\varpi^{-t}x\in(\mathcal{O}_{\widehat{K}_{\infty}}^{\flat})_{A/p}$
implies $[\varpi]^{-t}[x]\in W(\mathcal{O}_{\widehat{K}_{\infty}}^{\flat})_{A}$.
Replacing $x$ with $\varpi^{-t}x$, we reduce to the case $t=0$.
Using the continuity of $[\bullet]$, we may assume $x=\sum x_{i}\otimes e_{i}\in\mathcal{O}_{\widehat{K}_{\infty}}^{\flat}\otimes_{\F_{p}}A/p$.
But then it is clear that $[x]=\sum[x_{i}]\otimes e_{i}$ lies $W(\mathcal{O}_{\widehat{K}_{\infty}}^{\flat})\otimes_{\Z_{p}}A\subset W(\mathcal{O}_{\widehat{K}_{\infty}}^{\flat})_{A}$.

To show property 4 holds, we may again twist, so it suffices to show
that $p^{n}[x]\text{ mod }p^{n+1}\in\widetilde{\A}_{K,A/p^{n}}^{(0,r],+}$
implies $p^{n}[x]\in\widetilde{\A}_{K,A}^{(0,r],+}$. Again using
the continuity of $[\bullet]$, we may assume $x\in\widehat{K}_{\infty}^{\flat}\otimes_{\Z_{p}}A$.
Hence $[x]$, and consequently $p^{n}[x]$, belongs to the image of
$\widetilde{\A}_{K}^{(0,r]}\otimes_{\Z_{p}}A$, in $\widetilde{\A}_{K,A}^{(0,r]}$.
Explicitly, writing $x=\sum x_{i}\otimes e_{i}$ with $x_{i}\in\widehat{K}_{\infty}^{\flat}$,
the sum being finite, we have that $p^{n}[x]$ is the image under
the map $\widetilde{\A}_{K}^{(0,r]}\otimes_{\Z_{p}}A\rightarrow\widetilde{\A}_{K,A}^{(0,r]}$
of $\sum[x_{i}]\otimes p^{n}\widetilde{e}_{i}$. Taking this modulo
$p^{n+1}$, we obtain that $p^{n}[x]$ mod $p^{n+1}$ is the image
of $\sum[x_{i}]\otimes\pi_{n}(p^{n}\widetilde{e}_{i})$ under the
map $\widetilde{\A}_{K}^{(0,r]}\otimes_{\Z_{p}}p^{n}A/p^{n+1}\rightarrow\widetilde{\A}_{K,A/p^{n+1}}^{(0,r]}$.
On the other hand by assumption $p^{n}[x]$ mod $p^{n+1}\in\widetilde{\A}_{K,A/p^{n}}^{(0,r],+}$.
Since the $\pi_{n}(p^{n}\widetilde{e}_{i})$ are linearly independent,
the assumption $p^{n}[x]\text{ mod }p^{n+1}\in\widetilde{\A}_{K,A/p^{n}}^{(0,r],+}$
implies that each $[x_{i}]$ lies in the image of $\widetilde{\A}_{K}^{(0,r],\circ}$.
Hence $p^{n}[x]$, which is the image of $\sum[x_{i}]\otimes p^{n}\widetilde{e_{i}}$
under $\widetilde{\A}_{K}^{(0,r]}\otimes_{\Z_{p}}A\rightarrow\widetilde{\A}_{K,A}^{(0,r]}$,
has to lie in $\widetilde{\A}_{K,A}^{(0,r],+}$. This concludes the
proof. 
\end{proof}
The next key lemma is essentially the same as \cite[Lem. 1.3]{dSP19},
but we include the proof for completeness. Let $R$ be a commutative
ring endowed with a nonarchimedean valuation, where we allow $\val=-\infty$.
Suppose there exists $p\in\Z_{\geq2}$ and an invertible morphism
$\varphi:R\rightarrow R$ such that $p\val(x)=\val(\varphi(x))$.
Recall our notation for valuation of matrices from $\mathsection1.4$.
\begin{lem}
Let $X\in\GL_{d}(R)$. Then for any $c<(1/p-1)[\val(X)+\val(X^{-1})]$
and for any $Y\in\M_{d}(R)$ there exist $U,V\in\M_{d}(R)$ such that
$\val(V)\geq c$ and 
\[
X^{-1}\varphi(U)X-U=Y-V.
\]
\end{lem}

\begin{proof}
Let
\[
U=\sum_{i=1}^{N}\varphi^{-1}(X)\varphi^{-2}(X)\cdots\varphi^{-i}(X)\cdot\varphi^{-i}(Y)\cdot\varphi^{-i}(X)^{-1}\cdots\varphi^{-2}(X)^{-1}\varphi^{-1}(X)^{-1}
\]
and
\[
V=\varphi^{-1}(X)\varphi^{-2}(X)\cdots\varphi^{-N}(X)\cdot\varphi^{-N}(Y)\cdot\varphi^{-N}(X)^{-1}\cdots\varphi^{-2}(X)^{-1}\varphi^{-1}(X)^{-1}.
\]
Then 
\[
X^{-1}\varphi(U)X-U=Y-V.
\]

If $Y=0$ then choose $V=0$. Otherwise, by selecting $N$ large enough,
we can make $\val(\varphi^{-N}(Y))$ as close as we want to 0. In
addition,
\[
\val(\varphi^{-1}(X)\varphi^{-2}(X)\cdots\varphi^{-N}(X))\geq(p^{-1}+...+p^{-N})\val(X)
\]
and
\[
\val(\varphi^{-N}(X)^{-1}\cdots\varphi^{-2}(X)^{-1}\varphi^{-1}(X)^{-1})\geq(p^{-1}+...+p^{-N})\val(X^{-1}),
\]
so by taking $N$ sufficiently large we have $\val(V)\geq c$. 
\end{proof}
\begin{prop}
Let $\widetilde{M}$ be a free \textup{\emph{é}}tale $\varphi$-module
over $\widetilde{\A}_{K,A}$. Then there exists a free \textup{\emph{é}}tale
$\varphi$-module $\widetilde{M}^{\dagger}$ over $\widetilde{\A}_{K,A}^{\dagger}$,
contained in $\widetilde{M}$, such that the natural map 
\[
\widetilde{\A}_{K,A}\otimes_{\widetilde{\A}_{K,A}^{\dagger}}\widetilde{M}^{\dagger}\rightarrow\widetilde{M}
\]
is an isomorphism.
\end{prop}

\begin{proof}
Choose a basis of $\widetilde{M}$ and let $X=\Mat(\varphi)\in\GL_{d}(\widetilde{\A}_{K,A})$
be the matrix of $\varphi$ in this basis. We need to show there exists
a matrix $U\in\GL_{d}(\widetilde{\A}_{K,A})$ such that $C=U^{-1}X\varphi(U)\in\GL_{d}(\widetilde{\A}_{K,A}^{\dagger})$.
In fact we shall find $U$ so that $C\in\GL_{d}(\widetilde{\A}_{K,A}^{(0,\infty)})$.

To do this write $X=\sum_{n\geq0}p^{n}[X_{n}]$ with 
\[
X_{n}\in\M_{d}(\widetilde{\A}_{K,A}/p)=\M_{d}(\widetilde{\A}_{K,A/p})=\M_{d}(\widetilde{\A}_{K,A/p}^{(0,\infty)}),
\]
 and 
\[
\val_{A}^{(0,\infty)}(p^{k}[X_{n}])=\val_{p^{k}A/p^{k+1}}^{(0,\infty)}(p^{k}[X_{n}]\text{ mod }p^{k+1})
\]
 as we may according to Lemma 3.2. We shall construct $U=\sum_{n\geq0}p^{n}[U_{n}]$
with $U_{n}\in\M_{d}(\widetilde{\A}_{K,A}/p)$ and $\val_{A}^{(0,\infty)}([U_{n}])=\val_{A/p}^{(0,\infty)}(U_{n})$
by constructing $U_{n}$ inductively.

For $n=0$ take $U_{0}=\Id$. Now suppose $U_{0},...,U_{n-1}$ have
been defined. Let $U'=\sum_{i=0}^{n-1}p^{i}[U_{i}]$. We shall also
suppose we have chosen matrices $C_{0}=X_{0},C_{1},...,C_{n-1}\in\M_{d}(\widetilde{\A}_{K,A}/p)$
inductively such that
\[
U^{\prime-1}X\varphi(U^{\prime})\equiv\sum_{i=0}^{n-1}p^{i}[C_{i}]\text{ mod }p^{n}\M_{d}(\widetilde{\A}_{K,A}).
\]
Write 
\[
U^{\prime-1}X\varphi(U^{\prime})-\sum_{i=0}^{n-1}p^{i}[C_{i}]=p^{n}Y
\]
with $Y\in\M_{d}(\widetilde{\A}_{K,A})$. Now look for $U_{n}\in\M_{d}(\widetilde{\A}_{K,A}/p)$
such that 
\[
(U^{\prime}+p^{n}[U_{n}])^{-1}X\varphi(U^{\prime}+p^{n}[U_{n}])\equiv\sum_{i=0}^{n}p^{i}[C_{i}]\text{ mod }p^{n+1}
\]
with $\val_{A/p}^{(0,\infty)}(C_{n})$ bounded below. Noting that
$(U^{\prime}+p^{n}[U_{n}])^{-1}=U^{\prime-1}-p^{n}[U_{n}]$ mod $p^{n+1}$,
and letting $\overline{Y}$ denote the reduction of $Y$ mod $p\M_{d}(\widetilde{\A}_{K,A})$,
it suffices to solve the mod $p$ equation
\[
U_{n}-X_{0}\varphi(U_{n})X_{0}^{-1}=\overline{Y}X_{0}^{-1}-C_{n}X_{0}^{-1}
\]
in $\widetilde{\A}_{K,A/p}$.

By Lemma 3.3, this equation can be solved for both $U_{n}$ and $C_{n}$,
with $C_{n}$, hence also $p^{n}[C_{n}]$, with $\val_{A}^{(0,\infty)}(p^{n}[C_{n}])$
bounded independently of $n$. More precisely, we can solve for $U_{n}$
and $C_{n}$ with 
\[
\val_{A/p}^{(0,\infty)}(C_{n}X_{0}^{-1})\geq c
\]
for any $c<\frac{1}{p-1}[\val(X_{0})+\val(X_{0}^{-1})]$. Then we
may and do choose the $C_{n}$'s in a way such that for $n\geq1$
we have
\[
\val_{A}^{(0,\infty)}(p^{n}[C_{n}])\geq\val_{A}^{(0,\infty)}([C_{n}])\geq\val_{A/p}^{(0,\infty)}(C_{n})\geq\val_{A/p}^{(0,\infty)}(X_{0})+c.
\]
In particular, this bound is independent of $n$, the bound in fact
depending only on $X_{0}$. With these choices of the $C_{n}$, letting
$C=\sum_{n\geq0}[C_{n}]p^{n}$ and $U=\sum_{n\geq0}[U_{n}]p^{n}$
we therefore have $C=U^{-1}X\varphi(U),$where $C\in\GL_{d}(\widetilde{\A}_{K,A}^{(0,\infty)})$
because $\val_{A}^{(0,\infty)}(C)\geq c>-\infty$.
\end{proof}

\subsection{Descending $\varphi$-module morphisms}

We will need to regularize the action of $\varphi$.
\begin{prop}
Let $1/r\in\Z[1/p]_{>0}\cup\{\infty\}$, and let $X\in\M_{d}(\widetilde{\A}_{K,A}^{(0,r]}),Y\in\M_{e}(\widetilde{\A}_{K,A}^{(0,r]})$
and $U\in\M_{d\times e}(\widetilde{\A}_{K,A})$ satisfy
\[
\varphi(U)=XUY.
\]
Then $U\in\M_{d\times e}(\widetilde{\A}_{K,A}^{(0,r]})$.

In particular, if $X\in\M_{d}(\widetilde{\A}_{K,A}^{\dagger}),Y\in\M_{e}(\widetilde{\A}_{K,A}^{\dagger})$
and $U\in\M_{d\times e}(\widetilde{\A}_{K,A})$, then $U\in\M_{d\times e}(\widetilde{\A}_{K,A}^{\dagger})$.
\end{prop}

\begin{proof}
Write $X=[\varpi]^{t}X'$ with $\val_{A}^{(0,r]}(X')\geq0$, then
we see
\[
\val_{A}^{(0,r]}(\varphi^{-1}(X)\varphi^{-2}(X)\cdots\varphi^{-N}(X))\geq t(p^{-1}+...+p^{-N})
\]
is bounded independently of $N$. (This also happens with another
bound if $r=\infty$). A similar analysis applies to $\varphi^{-N}(Y)\cdots\varphi^{-2}(Y)\varphi^{-1}(Y)$.

From the equation $U=\varphi^{-1}(X)\varphi^{-1}(U)\varphi^{-1}(Y)$
we get by iteration
\[
U=\varphi^{-1}(X)\varphi^{-2}(X)\cdots\varphi^{-N}(X)\cdot\varphi^{-N}(U)\cdot\varphi^{-N}(Y)\cdots\varphi^{-2}(Y)\varphi^{-1}(Y)
\]
which we write as $U=X_{N}\varphi^{-N}(U)Y_{N}$, with $\val_{A}^{(0,r]}(X_{N})$
and $\val_{A}^{(0,r]}(Y_{N})$ bounded independently of $N$. 

Now write $U=\sum_{n\geq0}[U_{n}]p^{n}$, where $[\bullet]$ denotes
the generalized Teichmüller digit of Lemma 3.2, applied successively
to the entries of the reduction of $U$ mod $p^{n}$. Let $U^{(k)}$
be the $k$'th-truncation of $U$, i.e. $U^{(k)}=\sum_{n=0}^{k-1}[U_{n}]p^{n}$.
We then have
\[
U^{(k)}=X_{N}\varphi^{-N}(U^{(k)})Y_{N}\mod p^{k}.
\]
Fixing $k$ and choosing $N$ large we can make $\val_{A}^{(0,r]}(X_{N}\varphi^{-N}(U^{(k)})Y_{N})\geq c$
where $c$ is a constant depending only on $X$ and $Y$, because
$\val_{A}^{(0,r]}(\varphi^{-N}(U^{(k)}))=p^{-N}\val_{A}^{(0,r]}(U^{(k)})$.

We then get that 
\[
\val_{A/p^{k}}^{(0,r]}(U^{(k)}\text{ mod }p^{k})\geq\val_{A}^{(0,r]}(U^{(k)})
\]
is bounded below by $c$. 

We shall now prove by induction that $\val_{A}^{(0,r]}([U_{k-1}]p^{k-1})\geq c$.
By possibly making $c$ smaller, we may assume this is true for $k=1$.
Arguing by induction on $k$, using $U^{(k)}=[U_{k-1}]p^{k-1}+U^{(k-1)}$,
we deduce that 
\[
\val_{A/p^{k}}^{(0,r]}([U_{k-1}]p^{k-1}\text{ mod }p^{k})\geq c
\]
is bounded below by $c$. But by Lemma 3.2, this implies that $\val_{A}^{(0,r]}([U_{k-1}]p^{k-1})\geq c$.

Hence for each $k$, we have $\val_{A}^{(0,r]}(U^{(k)})\geq c$, and
consequently $\val_{A}^{(0,r]}(U)\geq c$. This proves that $U\in\M_{d\times e}(\widetilde{\A}_{A}^{(0,r]})$,
as required.
\end{proof}
\begin{cor}
Let $\widetilde{M}$ be a free \textup{\emph{é}}tale $\varphi$-module
over $\widetilde{\A}_{K,A}$. Then there exists a unique free \textup{\emph{é}}tale
$\varphi$-module $\widetilde{M}^{\dagger}$ over $\widetilde{\A}_{K,A}^{\dagger}$,
contained in $\widetilde{M}$, such that the natural map
\[
\widetilde{M}^{\dagger}\otimes_{\widetilde{\A}_{K,A}^{\dagger}}\widetilde{\A}_{K,A}\rightarrow\widetilde{M}
\]
is an isomorphism.
\end{cor}

\begin{proof}
The existence was established in Proposition 3.4, and the uniqueness
follows from Proposition 3.5.
\end{proof}
\begin{cor}
Let $\widetilde{M}$ be a free \textup{\emph{é}}tale $\varphi$-module
over $\widetilde{\A}_{A}$, and let $\widetilde{M}^{\dagger}$ be
the unique free \textup{\emph{é}}tale $\varphi$-module over $\widetilde{\A}_{A}^{\dagger}$
with $\widetilde{M}^{\dagger}\otimes_{\widetilde{\A}_{A}^{\dagger}}\widetilde{\A}_{A}\xrightarrow{\sim}\widetilde{M}$
as in Corollary 3.6. Then
\[
(\widetilde{M}^{\dagger})^{\varphi=1}=\widetilde{M}^{\varphi=1}.
\]
\end{cor}

\begin{proof}
Let $e_{1},...,e_{d}$ be a basis of $\widetilde{M}^{\dagger}$ and
let $X=\Mat(\varphi)\in\GL_{d}(\widetilde{\A}_{K,A}^{\dagger})$ be
the matrix of $\varphi$ with respect to this basis. If $m\in\widetilde{M}^{\varphi=1}$
and $m=\sum a_{i}e_{i}$ with $a_{i}\in\widetilde{\A}_{K,A}$, the
vector $a=(a_{i})$ satisfies the equation
\[
\varphi(a)=X^{-1}a,
\]
so by Proposition 3.5 we conclude that the $a_{i}$ belong to $\widetilde{\A}_{K,A}^{\dagger}$.
\end{proof}
\begin{rem}
Using Proposition 3.5, one can show that $\widetilde{M}^{\dagger}$
admits the following characterization: it is the subset of all elements
$m$ of $\widetilde{M}$ such that $\{\varphi^{i}(m)\}_{i\geq0}$
spans a finitely generated $\widetilde{\A}_{K,A}^{\dagger}$-submodule.
\end{rem}

\subsection{The equivalence of categories}

The following lemma will allow us to reduce from the projective case
to the free case.
\begin{lem}
Let $M$ be a projective \textup{\emph{é}}tale $\varphi$-module over
a ring $R$. Then there exists a free \textup{\emph{é}}tale $\varphi$-module
$F$ over $R$ such that $M$ is a direct summand of $F$.
\end{lem}

\begin{proof}
This argument of \cite[Lem. 5.2.14]{EG21} carries over.
\end{proof}
\begin{lem}
If $M,N$ are projective \textup{\emph{é}}tale $\varphi$-modules
over a ring $R$ then $M^{\vee}\otimes_{R}N$ is also a projective
\textup{\emph{é}}tale $\varphi$-module, and we have a natural identification
\[
\Hom_{R,\varphi}(M,N)\xrightarrow{\sim}(M^{\vee}\otimes_{R}N)^{\varphi=1}.
\]
\end{lem}

\begin{proof}
This is \cite[Lem. 2.5.4]{EG19}.
\end{proof}
\begin{prop}
Let $\widetilde{M}^{\dagger},\widetilde{N}^{\dagger}$ be projective
\textup{\emph{é}}tale $\varphi$-modules over $\widetilde{\A}_{K,A}^{\dagger}$,
and let $\widetilde{M}=\widetilde{M}^{\dagger}\otimes_{\widetilde{\A}_{K,A}^{\dagger}}\widetilde{\A}_{K,A}$,
$\widetilde{N}=\widetilde{N}^{\dagger}\otimes_{\widetilde{\A}_{K,A}^{\dagger}}\widetilde{\A}_{K,A}$.
Then
\[
\Hom_{\widetilde{\A}_{K,A}^{\dagger},\varphi}(\widetilde{M}^{\dagger},\widetilde{N}^{\dagger})\xrightarrow{\sim}\Hom_{\widetilde{\A}_{K,A},\varphi}(\widetilde{M},\widetilde{N}).
\]
\end{prop}

\begin{proof}
We argue as in \cite[Prop. 2.6.6]{EG19}. i.e. let $\widetilde{P}^{\dagger}=(\widetilde{M}^{\dagger})^{\vee}\otimes\widetilde{N}^{\dagger}$,
then $\widetilde{P}^{\dagger}$ is a projective étale $\varphi$-module
over $\widetilde{\A}_{K,A}^{\dagger}$. By Lemma 3.10, we need to
check that $(\widetilde{P}^{\dagger})^{\varphi=1}=\widetilde{P}^{\varphi=1}$,
where $\widetilde{P}=\widetilde{P}^{\dagger}\otimes_{\widetilde{\A}_{A}^{\dagger}}\widetilde{\A}_{A}$.
Since the formation of $\varphi$-invariants is compatible with direct
sums we reduce by Lemma 3.9 to the case that $\widetilde{P}^{\dagger}$
and $\widetilde{P}$ are free. But in this case the equality is known
by Corollary 3.7.
\end{proof}
\begin{prop}
Let $\widetilde{M}$ be a projective \textup{\emph{é}}tale $\varphi$-module
over $\widetilde{\A}_{K,A}$. Then there exists a projective \textup{\emph{é}}tale
$\varphi$-module $\widetilde{M}^{\dagger}$ over $\widetilde{\A}_{K,A}^{\dagger}$
and an isomorphism $\widetilde{M}^{\dagger}\otimes_{\widetilde{\A}_{K,A}^{\dagger}}\widetilde{\A}_{K,A}\xrightarrow{\sim}\widetilde{M}$.
\end{prop}

\begin{proof}
By Lemma 3.9, we may write $\widetilde{M}$ as a direct summand of
a free étale $\varphi$-module $\widetilde{F}$ over $\widetilde{\A}_{K,A}$.
By Proposition 3.4, there exists a free étale $\varphi$-module $\widetilde{F}^{\dagger}$
over $\widetilde{\A}_{K,A}^{\dagger}$ and an isomorphism $\widetilde{F}^{\dagger}\otimes_{\widetilde{\A}_{K,A}^{\dagger}}\widetilde{\A}_{K,A}\xrightarrow{\sim}\widetilde{F}$.
By Proposition 3.11, the idempotent in $\End(\widetilde{F})$ corresponding
to $\widetilde{M}$ (the projector onto $\widetilde{M}$) comes from
an idempotent in $\End(\widetilde{F}^{\dagger})$, and we may take
$\widetilde{M}^{\dagger}$ to be the étale $\varphi$-module corresponding
to this idempotent.
\end{proof}
Combining Propositions 3.11 and 3.12 we get a full descent result
for projective étale $\varphi$-modules:
\begin{thm}
The functor $\widetilde{M}^{\dagger}\mapsto\widetilde{M}$ induces
an equivalence of categories from the category of projective étale
$\varphi$-modules over $\widetilde{\A}_{K,A}^{\dagger}$ to the category
of projective étale $\varphi$-modules over $\widetilde{\A}_{K,A}$.
\end{thm}

Finally, to prove the descent result for $(\varphi,\Gamma_{K})$-modules,
we only need to descend $\Gamma_{K}$ along the equivalence:

\emph{Proof of Theorem 3.1. }For full faithfulness, take $\Gamma_{K}$-invariants
of both sides in the isomorphism of Proposition 3.11.

For essential surjectivity, suppose $\widetilde{M}$ is a projective
étale $(\varphi,\Gamma_{K})$-module over $\widetilde{\A}_{K,A}$,
and let $\widetilde{M}^{\dagger}$ be the unique étale projective
$\varphi$-module over $\widetilde{\A}_{K,A}^{\dagger}$ associated
to it in Proposition 3.12. Then by uniquness, $\widetilde{M}^{\dagger}$
is stable under the action of $\Gamma_{K}$, hence is actually a $(\varphi,\Gamma_{K})$-module
over $\widetilde{\A}_{K,A}^{\dagger}$. This concludes the proof.
$\hfill\ensuremath{\Box}$

\section{The Tate\textendash Sen method for Tate rings}

In this section, we present a variant of the Tate\textendash Sen method
of \cite{BC08} which will later allow us to avoid the use of Galois
representations when decompleting $(\varphi,\Gamma)$-modules. Furthermore,
it will apply in the case the coefficents have nonzero $p$-torsion.
The idea, inspired by \cite{Bel20}, is to replace the role of $p$
in the original Tate\textendash Sen method by that of a pseudouniformizer.
Technically speaking, the results presented here are neither a special
case nor a generalization of the setting in \cite{BC08}, because
of the difference in assumptions. However we are not aware of any
applications of \cite{BC08} where the method of this section would
not apply also.

We work in the following general setting. Let $G_{0}$ be a profinite
group endowed with a continuous character $\chi:G_{0}\rightarrow\Z_{p}^{\times}$
with open image and let $H_{0}=\ker\chi$. If $g\in G_{0}$, let $n(g)=\val_{p}(\chi(g)-1)\in\Z$.
For $G$ an open subgroup of $G_{0}$, set $H=G\cap H_{0}$. Let $G_{H}$
be the normalizer of $H$ in $G_{0}$. Note $G_{H}$ is open in $G_{0}$
since $G\subset G_{H}$. Finally let $\widetilde{\Gamma}_{H}=G_{H}/H$
and write $C_{H}$ for the center of $\widetilde{\Gamma}_{H}$. By
\cite[Lem. 3.1.1]{BC08} the group $C_{H}$ is open in $\widetilde{\Gamma}_{H}$.
Let $n_{1}(H)$ be the smallest positive integer such that $\chi(C_{H})$
contains $1+p^{n}\Z_{p}$.

Let $(\widetilde{\Lambda},\widetilde{\Lambda}^{+})$ be a pair of
topological rings with $\widetilde{\Lambda}^{+}\subset\widetilde{\Lambda}$,
and $f$ an element of $\widetilde{\Lambda}^{+}$.\textbf{ }We shall
make the following assumptions.

(i) $\widetilde{\Lambda}$ is a Tate ring, with $\widetilde{\Lambda}^{+}$
a ring of definition, and $f$ a pseudouniformizer.

(ii) $\widetilde{\Lambda}^{+}$ is $f$-adically complete.

(iii) There exists a valuation $\val_{\Lambda}:\widetilde{\Lambda}\rightarrow(-\infty,\infty]$
defining the topology on $\widetilde{\Lambda}$ such that $\val_{\Lambda}(fx)=\val_{\Lambda}(f)+\val_{\Lambda}(x)$
for $x\in\widetilde{\Lambda}$, and such that $\widetilde{\Lambda}^{+}=\widetilde{\Lambda}^{\val_{\Lambda}\geq0}$.

(iv) The group $G_{0}$ acts on $\widetilde{\Lambda}$, and this action
is unitary for the valuation $\val_{\Lambda}$.
\begin{lem}
If $U\in\M_{d}(\widetilde{\Lambda})$ has $\val_{\Lambda}(U-1)>0$
then $U\in\GL_{d}(\widetilde{\Lambda})$ with inverse $\sum_{n=0}^{\infty}(1-U)^{n}$.
\end{lem}

\subsection{The Tate\textendash Sen axioms}

With the previous setting, we define them to be the following.\footnote{The only difference from the usual conditions is in (TS2).}

(TS1) There exists $c_{1}>0$ such that for each pair $H_{1}\subset H_{2}$
of open subgroups of $H_{0}$ there exists $\alpha\in\widetilde{\Lambda}^{H_{1}}$
such that $\val_{\Lambda}(\alpha)>-c_{1}$ and $\sum_{\tau\in H_{2}/H_{1}}\tau(\alpha)=1$.

(TS2) There exists $c_{2}>0$ and for each open subgroup $H$ of $H_{0}$
an integer $n(H)$, as well as an increasing sequence $(\Lambda_{H,n})_{n\geq n(H)}$
of closed subalgebras of $\widetilde{\Lambda}^{H}$, each containing\textbf{
$f^{\pm1}$},\textbf{ }and $\Lambda_{H,n}$-linear maps
\[
\R_{H,n}:\widetilde{\Lambda}^{H}\rightarrow\Lambda_{H,n}
\]
such that

(1) If $H_{1}\subset H_{2}$ then $\Lambda_{H_{2},n}\subset\Lambda_{H_{1},n}$
and $\R_{H_{1},n}|_{\widetilde{\Lambda}^{H_{2}}}=\R_{H_{2},n}$.

(2) $\R_{H,n}(x)=x$ if $x\in\Lambda_{H,n}$.

(3) $g(\Lambda_{H,n})=\Lambda_{gHg^{-1},n}$ and $g(\R_{H,n}(x))=\R_{gHg^{-1}}(gx)$
if $g\in G_{0}$.

(4) If $n\geq n(H)$ and if $x\in\widetilde{\Lambda}^{H}$ then $\val_{\Lambda}(\R_{H,n}(x))\geq\val_{\Lambda}(x)-c_{2}$.

(5) If $x\in\widetilde{\Lambda}^{H}$ then $\lim_{n\rightarrow\infty}\R_{H,n}(x)=x$.

(TS3) There exists $c_{3}>0$ and, for each open subgroup $G$ of
$G_{0}$ an integer $n(G)\geq n_{1}(H)$ where $H=G\cap H_{0}$, such
that if $n(\gamma)\leq n\leq n(G)$ for $\gamma\in\widetilde{\Gamma}_{H}$
then $\gamma-1$ is invertible on $X_{H,n}=(1-\R_{H,n})(\widetilde{\Lambda}^{H})$
and $\val_{\Lambda}((\gamma-1)^{-1}(x))\geq\val_{\Lambda}(x)-c_{3}$.

For the rest of the section we shall assume (TS1), (TS2) and (TS3)
are satisfied.
\begin{rem}
If $H_{0}$ is trivial then the conditions simplify, and in particular,
(TS1) is automatically satisfied for any $c_{1}>0$. This will be
the setting in $\mathsection5$, however, we produce here a more general
framework for ease of future applications.
\end{rem}

\subsection{Descent to $H$-invariants}
\begin{lem}
Let $H$ be an open subgroup of $H_{0}$, $a>c_{1}$. Suppose $\tau\mapsto U_{\tau}$
is a continuous 1-cocycle of $H$ valued in $\GL_{d}(\widetilde{\Lambda})$
which verifies and $\val(U_{\tau}-1)\geq a$ for each $\tau\in H$.
Then there exists a matrix $M\in\GL_{d}(\widetilde{\Lambda})$ with
$\val_{\Lambda}(M-1)\geq a-c_{1}$ such that the cocycle $\tau\mapsto M^{-1}U_{\tau}\tau(M)$
verifies $\val_{\Lambda}(M^{-1}U_{\tau}\tau(M)-1)\geq a+1$.
\end{lem}

\begin{proof}
This is proven in the same way as \cite[Lem. 3.2.1]{BC08}. (The analogue
of the condition $M-1\in p^{k}\M_{d}(\widetilde{\Lambda})$, which
appears in loc. cit., is $M-1\in f^{k}\M_{d}(\widetilde{\Lambda})$.
It is vacuous because with our assumptions $f$ is invertible in $\widetilde{\Lambda}$).
\end{proof}
\begin{cor}
Let $H$ be an open subgroup of $H_{0}$, $a>c_{1}$ and $\tau\mapsto U_{\tau}$
a continuous 1-cocycle of $H$ valued in $\GL_{d}(\widetilde{\Lambda})$.
Suppose $\val(U_{\tau}-1)\geq a$ for each $\tau\in H$. Then there
exists $M\in\GL_{d}(\widetilde{\Lambda})$ such that $\val_{\Lambda}(M-1)\geq a-c_{1}$
such that the cocycle $\tau\mapsto M^{-1}U_{\tau}\tau(M)$ is trivial.
\end{cor}

\begin{proof}
This is proven in the same way as \cite[Cor. 3.2.2]{BC08}.
\end{proof}

\subsection{Decompletion}

The following lemma needs to be slightly modified compared to the
treatment of \cite{BC08}.
\begin{lem}
Let $\delta>0$ and let $a,b\in\bbR$ such that $a\geq c_{2}+c_{3}+\delta$
and $b\ge\sup(a+c_{2},2c_{2}+2c_{3}+\delta$). Let $H$ be an open
subgroup of $H_{0}$, $n\geq n(H)$ and $\gamma\in\widetilde{\Gamma}_{H}$
such that $n(\gamma)\leq n$. Finally, let 
\[
U=1+f^{k}U_{1}+f^{k}U_{2}
\]
such that
\[
U_{1}\in\M_{d}(\Lambda_{H,n}),\val_{\Lambda}(U_{1})\geq a-\val_{\Lambda}(f^{k})
\]
\[
U_{2}\in\M_{d}(\widetilde{\Lambda}^{H}),\val_{\Lambda}(U_{2})\geq b-\val_{\Lambda}(f^{k}).
\]
Then there exists $M\in\M_{d}(\widetilde{\Lambda}^{H})$ with $\val(M-1)\geq b-c_{2}-c_{3}$
such that $M^{-1}U\gamma(M)=1+f^{k}V_{1}+f^{k}V_{2}$ such that
\[
V_{1}\in\M_{d}(\Lambda_{H,n}),\val_{\Lambda}(V_{1})\geq a-\val_{\Lambda}(f^{k})
\]
\[
V_{2}\in\M_{d}(\widetilde{\Lambda}^{H}),\val_{\Lambda}(V_{2})\geq b-\val_{\Lambda}(f^{k})+\delta.
\]
\end{lem}

\begin{proof}
This again just follows the proof of \cite[Lem. 3.2.3]{BC08}, but
there are a few small modifications required, so we give more details.
One sets\footnote{Compare this to the proof of loc. cit, where one takes $V=(1-\gamma)^{-1}(1-\R_{H,n})(U_{2})$.
The change is necessary for us because in general $\gamma-1$ does
not act linearly on $f^{k}$. This is also the only place where we
have used the assumption that $f^{\pm1}\in\Lambda_{H,n}$.)}
\[
V=f^{-k}(1-\gamma)^{-1}(1-\R_{H,n})(f^{k}U_{2})
\]
\[
V_{1}=U_{1}+\R_{H,n}(U_{2})
\]
\[
V_{2}=f^{-k}[(1+f^{k}V)^{-1}U\gamma(1+f^{k}V)-(1+f^{k}U_{1}+f^{k}\R_{H,n}(U_{2}))]
\]
\[
M=1+f^{k}V.
\]

By explicit calculations, one checks using (TS2), (TS3), the $\Lambda_{H,n}$-linearity
of $\R_{H,n}$, and the expansion 
\[
(1+f^{k}V)^{-1}=1-f^{k}V+f^{2k}V^{2}-...
\]
that all the terms are well defined and that the conditions are satisifed.
(Here we have implicitly used the assumption $f\in\Lambda_{H,n}$
in use of the $\Lambda_{H,n}$-linearity of $\R_{H,n}$.)
\end{proof}
\begin{cor}
Let $\delta>0$, $b\geq2c_{2}+2c_{3}+\delta$ and $n\geq n(H)$. If
$U\in\M_{d}(\widetilde{\Lambda}^{H})$ has $\val_{\Lambda}(U-1)\geq b$
then there exists $M\in\M_{d}(\widetilde{\Lambda}^{H})$ with $\val_{\Lambda}(M-1)\geq b-c_{2}-c_{3}$
such that $M^{-1}U\gamma(M)\in M_{d}(\Lambda_{H,n})$.
\end{cor}

\begin{proof}
This is the same proof as that of \cite[Cor. 3.2.4]{BC08}.
\end{proof}
\begin{lem}
Let $H$ be an open subgroup of $H_{0}$ and let $n\geq n(H)$, $\gamma\in\Gamma_{H}$
such that $n(\gamma)\leq n$ and $B\in\M_{l\times d}(\widetilde{\Lambda}^{H})$
be a matrix. Suppose there are $V_{1}\in\GL_{l}(\Lambda_{H,n}),V_{2}\in\GL_{d}(\Lambda_{H,n})$
such that $\val(V_{1}-1),\val(V_{2}-1)>c_{3}$ and $\gamma(B)=V_{1}BV_{2}$.
Then $B\in\M_{l\times d}(\Lambda_{H,n})$.
\end{lem}

\begin{proof}
The proof is exactly the same as that of \cite[Lem. 3.2.5]{BC08}.
The only difference between that lemma and the statement appearing
here is that there one further assumes $l=d$ and $B\in\GL_{d}(\Lambda_{H,n})$,
but these assumptions are not used in the proof.
\end{proof}

\subsection{Descent}
\begin{prop}
Let $\sigma\mapsto U_{\sigma}$ be a continuous 1-cocycle of $G_{0}$
valued in $\GL_{d}(\widetilde{\Lambda})$. If $G$ is an open subgroup
of $G_{0}$ such that $\val(U_{\sigma}-1)>c_{1}+2c_{2}+2c_{3}$ when
$\sigma\in G$, and if $H=G\cap H_{0}$, then there exists $M\in\M_{d}(\widetilde{\Lambda})$
with $\val(M-1)>c_{2}+c_{3}$ such that the 1-cocycle of $G_{0}$
given by $\sigma\mapsto V_{\sigma}=M^{-1}U_{\sigma}\sigma(M)$ is
trivial on $H$ and valued in $\GL_{d}(\Lambda_{H,n(G)}).$
\end{prop}

\begin{proof}
This is the same proof as that of \cite[Prop. 3.2.6]{BC08}.
\end{proof}
Let $M^{+}$ be a finite free $\widetilde{\Lambda}^{+}$-semilinear
representation of $G_{0}$ and for $H\subset H_{0}$ open let $\Lambda_{H,n}^{+}=\widetilde{\Lambda}^{+}\cap\Lambda_{H,n}$.
\begin{prop}
Suppose that $G$ is an open subgroup of $G_{0}$ and that $M^{+}$
has a basis such that $\val(\Mat(g)-1)>c_{1}+2c_{2}+2c_{3}$ for $g\in G$.
Let $H=G\cap H_{0}$.

Then for $n\geq n(G)$ there exists a unique free $\Lambda_{H,n}^{+}$-submodule
$\D_{H,n}^{+}(M^{+})$ of $M^{+}$ such that 

(1) $\D_{H,n}^{+}(M^{+})$ is fixed by $H$ and stable by $G_{0}$.

(2) The natural map $\widetilde{\Lambda}^{+}\otimes_{\Lambda_{H,n}^{+}}\D_{H,n}^{+}(M^{+})\rightarrow M^{+}$
is an isomorphism. In particular, $\D_{H,n}^{+}(M^{+})$ is free of
rank = $\rank M^{+}$.

(3) $\D_{H,n}^{+}(M^{+})$ has a basis which is $c_{3}$-fixed by
$G/H$, meaning that for $\gamma\in G/H$ we have $\val(\Mat(\gamma)-1)>c_{3}$.
\end{prop}

\begin{proof}
We follow the proof of \cite[Prop. 3.3.1]{BC08}, which is closely
related.

Let $v_{1},...,v_{d}$ be a basis of $M^{+}$ over $\widetilde{\Lambda}^{+}$.
We get a cocycle $U$ in $\H^{1}(G_{0},\GL_{d}(\widetilde{\Lambda}^{+}))$.
By assumption $\val(U_{\sigma}-1)>c_{1}+2c_{2}+2c_{3}$ if $\sigma\in G$.
By Proposition 4.8 there exists $M\in\M_{d}(\widetilde{\Lambda})$
such that $\val(M-1)>c_{2}+c_{3}$ and the cocycle $\sigma\mapsto V_{\sigma}=M^{-1}U_{\sigma}\sigma(M)$
is trivial on $H$ and is valued in $\GL_{d}(\Lambda_{H,n(G)})$.
Now $\val(M-1)>c_{2}+c_{3}>0$ so $M\in\GL_{d}(\widetilde{\Lambda}^{+})$.
It follows that $V$ is valued in $\GL_{d}(\Lambda_{H,n(G)})\cap\GL_{d}(\widetilde{\Lambda}^{+})=\GL_{d}(\Lambda_{H,n(G)}^{+})$.

Now let $M=(m_{ij})$ and $U_{\sigma}=(u_{ij})$. If we write $e_{i}=Mv_{i}$
for $i=1,...,d$ then
\[
\sigma(e_{k})=\sum_{j}\sigma(m_{jk})\sigma(v_{j})=\sum_{i}\sum_{j}u_{ij}\sigma(m_{ij})v_{i}=e_{k}
\]
for $\sigma\in H$. So $e_{1},...,e_{d}$ is a basis of $M^{+}$ fixed
by $H$, and if we write $\D_{H,n}^{+}(M^{+})=\oplus\Lambda_{H,n}^{+}e_{i}$
then the natural map $\widetilde{\Lambda}^{+}\otimes_{\Lambda_{H,n}^{+}}\D_{H,n}^{+}(M^{+})\rightarrow M^{+}$
is an isomorphism.

Further, if $\gamma\in G/H$ then $W=\Mat(\gamma)$ in the basis $e_{1},...,e_{d}$
is of the form $M^{-1}U_{\sigma}\sigma(M)$ where $\sigma\in G$ is
a lift of $\gamma$. Using the identity
\[
W-1=M^{-1}U_{\sigma}\sigma(M)-1=(M^{-1}-1)U_{\sigma}\sigma(M)+(U_{\sigma}-1)\sigma(M)+\sigma(M)-1,
\]
it follows that $\val(W-1)>c_{2}+c_{3}>c_{3}$. This implies that
the basis$\{e_{1},...,e_{d}\}$ is $c_{3}$-fixed. 

Finally, we show that $\D_{H,n}^{+}(M^{+})$ is unique. Choose $\gamma\in C_{H}$
with $n(\gamma)=n$ and let $e_{1}',...,e_{d}'$ be another basis.
Then $\Mat_{\{e_{i}\}}(\gamma)=W$ and $\Mat_{\{e_{i}'\}}(\gamma)=W'$
are both in $\GL_{d}(\Lambda_{H,n(G)}^{+})$ with $n\geq n(G)$ and
$\val(W-1),\val(W'-1)>c_{3}$. Let $B\in\GL_{d}(\widetilde{\Lambda}^{+})$
be the matrix expressing $e_{i}'$ in terms of $e_{i}$. Then $B$
is $H$-invariant and $W'=B^{-1}W\gamma(B)$. According to Lemma 4.7
we have $B\in\GL_{d}(\Lambda_{H,n})$, so $B\in\GL_{d}(\widetilde{\Lambda}^{+})\cap\GL_{d}(\Lambda_{H,n})=\GL_{d}(\Lambda_{H,n}^{+})$.
It follows that $e_{i}'$ and $e_{i}$ generate the same submodule.
\end{proof}
\begin{prop}
With the notations of the previous proposition, the module $\D_{H,n}^{+}(M^{+})$
admits the following characterization: it is the union of all finitely
generated $\Lambda_{H,n}^{+}$-submodules of $M^{+}$ which are stable
by $G_{0}$, fixed by $H$ and which are generated by a $c_{3}$-fixed
set of generators.
\end{prop}

\begin{proof}
Indeed, if we have a submodule generated by $c_{3}$-fixed elements
$f_{1},...,f_{l}$ and if $e_{1},...,e_{d}$ is a $c_{3}$-fixed basis,
let $B\in\M_{l\times d}(\widetilde{\Lambda}^{H,+})$ be a matrix expressing
the $f_{i}$ in terms of the $e_{i}$. We have
\[
\Mat_{\{f_{i}\}}(\gamma)B=\gamma(B)\Mat_{\{e_{i}\}}(\gamma).
\]
(Here $\Mat_{\{f_{i}\}}(\gamma)$ is not necessarily a uniquely determined
matrix, since the submodule generated by the $f_{i}$ may not be free,
but this does not matter). We have $\val(\Mat_{\{f_{i}\}}(\gamma)-1)>c_{3}$,
and this implies that $\Mat_{\{f_{i}\}}(\gamma)$ is invertible. So
by Lemma 4.7,
\[
B\in\M_{l\times d}(\Lambda_{H,n})\cap\M_{l\times d}(\widetilde{\Lambda}^{H,+})=\M_{l\times d}(\Lambda_{H,n}^{+}),
\]
which means the submodule generated by $f_{i}$ is contained in $\D_{H,n}^{+}(M^{+})$. 
\end{proof}
We introduce two additive, $\otimes$-categories:
\begin{itemize}
\item $\Mod_{\widetilde{\Lambda}^{+}}^{G_{0}}(G)$, the category of finite
free $\widetilde{\Lambda}^{+}$-semilinear representations of $G_{0}$
such that for some basis $\val(\Mat(g)-1)>c_{1}+2c_{2}+2c_{3}$ for
$g\in G$.
\item $\Mod_{\Lambda_{H,n}^{+}}^{G_{0}}(G)$, the category of finite free
$\Lambda_{H,n}^{+}$-semilinear representations of $G_{0}$ that are
fixed by $H=G\cap H_{0}$ and which have $c_{3}$-fixed basis. 
\end{itemize}
For $n\geq n(G)$, Propositions 4.9 and 4.10 give us a functor $M^{+}\mapsto\D_{H,n}^{+}(M^{+})$
from $\Mod_{\widetilde{\Lambda}^{+}}(G)$ to $\Mod_{\Lambda_{H,n}^{+}}(G)$.
We also have a functor $N^{+}\mapsto\widetilde{\Lambda}^{+}\otimes_{\Lambda_{H,n}^{+}}N^{+}$
from $\Mod_{\Lambda_{H,n}^{+}}(G)$ to $\Mod_{\widetilde{\Lambda}^{+}}(G)$.
We can then express Proposition 4.9 in the following way.
\begin{thm}
The functor $M^{+}\mapsto\D_{H,n}^{+}(M^{+})$ gives an equivalence
of categories between $\Mod_{\widetilde{\Lambda}^{+}}^{G_{0}}(G)$
and $\Mod_{\Lambda_{H,n}^{+}}^{G_{0}}(G)$ for $n\geq n(G)$.
\end{thm}

\section{Deperfection}

In this section we explain how to descend $(\varphi,\Gamma_{K})$-modules
from $\widetilde{\A}_{K,A}^{\dagger}$ to $\A_{K,A}^{\dagger}$.

\subsection{Verifying the Tate\textendash Sen axioms}

We take $G_{0}=\Gamma_{K}$, and $\chi:G_{0}\rightarrow\Z_{p}^{\times}$
is the cyclotomic character, so that $H_{0}=1$.

Let $\widetilde{\Lambda}=\widetilde{\A}_{K,A}^{(0,r]}$ for $1/r\in\Z[1/p]_{>0}$
with $r<1$ and $\widetilde{\Lambda}^{+}=\widetilde{\A}_{K,A}^{(0,r],+}$.
Then

(i) $\widetilde{\Lambda}$ is a Tate ring, with ring of definition
$\widetilde{\Lambda}^{+}$ , with a pseudouniformizer $f$ taken to
be $T$, the element introduced in $\mathsection2.1$. Note $T$ was
introduced as an element of $\A_{K,A}^{+}$, but it can be thought
of as an element of $\A_{K,A}^{(0,r]}$, the latter being a subring
of $\widetilde{\Lambda}$.

(ii) $\widetilde{\Lambda}^{+}$ is $T$-adically complete, by Proposition
2.3.

(iii) By $\mathsection6$ of \cite{Co08}, $T/[\varpi]$ is a unit
in $\widetilde{\A}^{(0,r],\circ}$ if $r<1$, hence also in $\widetilde{\A}_{K,A}^{(0,r],+}$.
We endow $\widetilde{\A}_{K,A}^{(0,r]}$ with the valuation $\val^{(0,r]}$:
\[
\val^{(0,r]}(a)=(p/p-1)\sup\{x\in\Z[1/p]:a\in[\varpi]^{x}\widetilde{\A}_{K,A}^{(0,r],+}\}.
\]
Note that $[\varpi]^{x}\widetilde{\A}_{A}^{(0,r],+}=T^{x}\widetilde{\A}_{A}^{(0,r],+}$
whenever $T^{x}$ makes sense. It therefore induces the $T$-adic
topology. 

(iv) The group $G_{0}$ acts continuously on $\widetilde{\A}_{K,A}^{(0,r]}$,
and is unitary for the valuation $\val^{(0,r]}$. 

As explained in Remark 4.2, the condition (TS1) is automatic in this
setting since $H_{0}=1$.

\textbf{The axiom (TS2)}. We shall check this axiom holds in the following
setting. Recalling $H_{0}=1$, we set
\[
\Lambda_{n}=\Lambda_{H_{0},n}:=\varphi^{-n}(\A_{K,A}^{(0,rp^{-n}]}),
\]
which is a closed subalgebra of $\widetilde{\A}_{K,A}^{(0,r]}$, and
\[
\R_{n}:=\R_{H_{0},n}:\widetilde{\A}_{K,A}^{(0,r]}\rightarrow\varphi^{-n}(\A_{K,A}^{(0,rp^{-n}]})
\]
is the continuous extension of the map $(\widetilde{\A}_{K}^{(0,r]})\rightarrow\varphi^{-n}(\A_{K}^{(0,rp^{-n}]})$
constructed in $\mathsection8$ of \cite{Co08} for the case $A=\Z_{p}$.
This extension exists because this original map is $T$-adically continuous.

We shall now verify (TS2) holds. The only thing which is not immediate
is condition 4, i.e. we need to check that there exists $c_{2}>0$
such that $\val_{\Lambda}(R_{n}(x))\geq\val_{\Lambda}(x)-c_{2}$.
To do this, choose $c_{2}>0$ which works in the case $A=\Z_{p}$,
which is known to exist by \cite[Prop. 4.2.1]{BC08}.

Suppose $x\in\widetilde{\A}_{K,A}^{(0,r]}$. If $\val_{\Lambda}(x)\geq(p/p-1)t$
for $t\in\Z[1/p]$ then $x\in[\varpi]^{t}\widetilde{\A}_{K,A}^{(0,r],+}$.
We may write $x$ as the image of a (possibly infinite) sum
\[
\sum x_{i}\otimes a_{i}
\]
with 
\[
x_{i}\in[\varpi]^{t}\widetilde{\A}_{K}^{(0,r],\circ},a_{i}\in A.
\]
so that $\val^{(0,r]}(x_{i})\geq(p/p-1)t.$

Now $\R_{n}(x)=\sum\R_{n}(x_{i})\otimes a_{i},$ and we have from
the case $A=\Z_{p}$ that 
\[
\val^{(0,r]}(\R_{n}(x_{i}))\geq(p/p-1)t-c_{2}.
\]
Hence 
\[
\R_{n}(x_{i})\in([\varpi]^{t-\frac{p-1}{p}c_{2}}\widetilde{\A}_{K}^{(0,r],\circ})\widehat{\otimes}A
\]
for each $i$, which shows that
\[
\R_{n}(x)\in[\varpi]^{t-\frac{p-1}{p}c_{2}}\widetilde{\A}_{K,A}^{(0,r],+},
\]
so that $\val_{\Lambda}(\R_{n}(x))\geq\frac{p}{p-1}t-c_{2}$. Hence
(TS2) holds with the same $c_{2}$.

\textbf{The axiom (TS3)}. We need to show that there exists $c_{3}>0$
and, for each open subgroup $G$ of $G_{0}$ an integer $n(G)$ such
that if $n\geq n(G)$ and if $n(\gamma)\leq n$ then $\gamma-1$ is
invertible on $X_{n}=(1-\R_{n})(\widetilde{\Lambda})$ and $\val_{\Lambda}((\gamma-1)^{-1}(x))\geq\val_{\Lambda}(x)-c_{3}$.

We shall show that if $c_{2}=c_{2}(\Z_{p})$ and $c_{3}=c_{3}(\Z_{p})$
work for the case $A=\Z_{p}$ as in \cite[Prop. 4.2.1]{BC08} then
any $c_{3}'>c_{2}(\Z_{p})+c_{3}(\Z_{p})$ works for general $A$.

Let $x\in X_{n}=(1-\R_{n})(\widetilde{\Lambda})$. Let $y\in\widetilde{\Lambda}\cong\widetilde{\A}_{K,A}^{(0,r]}$.
We may write $y$ as the image of
\[
\sum y_{i}\otimes a_{i}
\]
in $\widetilde{\A}_{K,A}^{(0,r]}$, with $y_{i}\in\widetilde{\A}_{K}^{(0,r]}$
(the sum possibly infinite). If $\val^{(0,r]}(y)\geq(p/p-1)t$ for
$t\in\Z[1/p]$ then $y\in[\varpi]^{t}\widetilde{\A}_{K,A}^{(0,r],+}$,
which is the image of $[\varpi]^{t}\widetilde{\A}_{K}^{(0,r],\circ}\widehat{\otimes}A$,
so one can choose $y_{i}\in[\varpi]^{t}\widetilde{\A}_{K}^{(0,r],\circ}$.
This shows that we may assume $\val^{(0,r]}(y_{i})\geq\val^{(0,r]}(y)$
for every $i$.

We let $x=(1-\R_{n})(y)$, so that $x$ is the image of $\sum(1-\R_{n})(y_{i})\otimes a_{i}$.
Writing $x_{i}=(1-\R_{n})(y_{i})$ we have 
\[
\val_{\Lambda}((\gamma-1)^{-1}(x_{i}))\geq\val_{\Lambda}(x_{i})-c_{3}\geq\val^{(0,r]}(y_{i})-c_{3},
\]
which approaches zero. The sum $\sum((\gamma-1)^{-1}(x_{i}))\otimes a_{i}$
therefore converges in $\widetilde{\A}_{K,A}^{(0,r]}$. This shows
that $\gamma-1$ is invertible on $X_{n}$. 

Finally, suppose $x\in X_{n}$ with $\val^{(0,r]}(x)\geq(p/p-1)t$,
we shall show that 
\[
\val^{(0,r]}((\gamma-1)^{-1}(x))\geq\frac{p}{p-1}t-c_{2}-c_{3}.
\]
By assumption, $x\in[\varpi]^{t}\widetilde{\A}_{K,A}^{(0,r],+}$,
so we may write $x=\sum x_{i}\otimes a_{i}$ (the sum possibly infinite)
with $x_{i}\in[\varpi]^{t}\widetilde{\A}_{K}^{(0,r],\circ}.$ Then
since $(1-\R_{n})$ is idempotent we have $x=\sum(1-\R_{n})(x_{i})\otimes a_{i}$.
Letting $y_{i}=(1-\R_{n})(x_{i})$ we have $y_{i}\in[\varpi]^{t-\frac{p-1}{p}c_{2}}\widetilde{\A}_{K}^{(0,r],\circ}$
and so $(\gamma-1)^{-1}(x)=\sum(\gamma-1)^{-1}(y_{i})\otimes a_{i}$.
Then 
\[
(\gamma-1)^{-1}(y_{i})\in[\varpi]^{t-\frac{p-1}{p}c_{2}-\frac{p-1}{p}c_{3}}\widetilde{\A}_{K}^{(0,r],\circ}
\]
 which shows 
\[
\val_{\Lambda}((\gamma-1)^{-1}(x))\geq\val_{\Lambda}(x)-c_{2}-c_{3},
\]
which shows we can take $c_{3}'=c_{2}-c_{3}$, as required.
\begin{rem}
$i$. With a little more work, once can prove that $\widetilde{\A}_{A}^{(0,r]}$
also satisfies the Tate\textendash Sen axioms. This recovers overconvergence
results appearing in $\mathsection4.2$ of \cite{BC08}. 

$ii$. The following was pointed out to us by Rebecca Bellovin: if
$A$ is the ring of definition of a pseudoaffinoid algebra with pseudouniformizer
$u$, the rings $\widetilde{\A}_{A/u}^{(0,r]}$ are consistent with
the reduction mod $u$ of the rings $\widetilde{\Lambda}_{A,(0,r]}$
of \cite{Bel20}. By taking $u$-adic limits, one should be able to
recover the main result of \cite{Bel20} from the results of $\mathsection4$.

$iii$. More generally one could probably phrase the results of this
section as some sort of stablity of our version of the Tate\textendash Sen
axioms under base change, but we have not attempted to do so.
\end{rem}

\subsection{Descent}

The following proposition is a variant of \cite[Thm. I.3.3]{Ber08}.
\begin{prop}
If $\widetilde{M}^{\dagger}$ is a projective \textup{\emph{é}}tale
$\varphi$-module over $\widetilde{\A}_{K,A}^{\dagger}$ then for
every $1/r\in\Z[1/p]_{>0}$ there exists a unique projective $\widetilde{\A}_{K,A}^{(0,r]}$-submodule
$\widetilde{M}^{(0,r]}\subset\widetilde{M}^{\dagger}$ such that 

i. The natural map $\widetilde{\A}_{K,A}^{\dagger}\otimes_{\widetilde{\A}_{K,A}^{(0,r]}}\widetilde{M}^{(0,r]}\rightarrow\widetilde{M}^{\dagger}$
is an isomorphism.

ii. $\varphi$ sends $\widetilde{M}^{(0,r]}$ into $\widetilde{\A}_{K,A}^{(0,r/p]}\otimes_{\widetilde{\A}_{K,A}^{(0,r]}}\widetilde{M}^{(0,r]}$,
and the induced map
\[
\widetilde{\A}_{K,A}^{(0,r/p]}\otimes_{\widetilde{\A}_{K,A}^{(0,r]}}^{\varphi}\widetilde{M}^{(0,r]}=\varphi^{*}\widetilde{M}^{(0,r]}\rightarrow\widetilde{\A}_{K,A}^{(0,r/p]}\otimes_{\widetilde{\A}_{K,A}^{(0,r]}}\widetilde{M}^{(0,r]}
\]
is an isomorphism. 

In particular,

1. $\widetilde{\A}_{K,A}^{(0,s]}\otimes_{\widetilde{\A}_{K,A}^{(0,r]}}\widetilde{M}^{(0,r]}=\widetilde{M}^{(0,s]}$
for $s<r$; 

2. $\varphi$ induces an isomorphism $\varphi^{*}\widetilde{M}^{(0,r]}\xrightarrow{\sim}\widetilde{M}^{(0,r/p]}$;

3. If $\widetilde{M}^{\dagger}$ is a $(\varphi,\Gamma_{K})$-module,
then each $\widetilde{M}^{(0,r]}$ is $\Gamma_{K}$-stable.

Furthermore, if $\widetilde{M}^{\dagger}$ is free, then so is $\widetilde{M}^{(0,r]}$.
\end{prop}

\begin{proof}
We start by proving the statement in the case where $\widetilde{M}^{\dagger}$
is free.

To prove existence, choose any basis $e_{1},..,e_{d}$ of $\widetilde{M}^{\dagger}$
over $\widetilde{\A}_{K,A}^{\dagger}$. Then $\Mat(\varphi)\in\GL_{d}(\widetilde{\A}_{K,A}^{\dagger})$,
and so for some $r_{0}$ and all $r\leq r_{0}$ we have $\Mat(\varphi),\Mat(\varphi^{-1})\in\GL_{d}(\widetilde{\A}_{K,A}^{(0,r_{0}]})$.
We take $\widetilde{M}^{(0,r]}=\bigoplus\widetilde{\A}_{K,A}^{(0,r]}e_{i}$,
so that $i$ and $ii$ are satisfied. For $r>r_{0}$ we may recursively
define $\widetilde{M}^{(0,r]}:=(\varphi^{-1})^{*}(\widetilde{M}^{(0,r/p]})$
using the isomorphism $(\varphi^{-1})^{*}\widetilde{M}^{\dagger}\xrightarrow{\sim}\widetilde{M}^{\dagger}$. 

For uniqueness, suppose that $\widetilde{M}^{(0,r],(1)}$ and $\widetilde{M}^{(0,r],(2)}$
are two submodules satisfying conditions $i$ and $ii$. Let $X\in\GL_{d}(\widetilde{\A}_{K,A}^{\dagger})$
be the transition matrix between the two bases of $\widetilde{M}^{(0,r],(1)}$
and $\widetilde{M}^{(0,r],(2)}$ and let $P_{1}$ and $P_{2}$ be
the matrices in $\GL_{d}(\widetilde{\A}_{K,A}^{(0,r/p]})$ of $\varphi$
in these bases. Then we have the equation
\[
X=P_{1}^{-1}\varphi(X)P_{2},
\]
which implies by Proposition 3.5 that $X\in\GL_{d}(\widetilde{\A}_{K,A}^{(0,r/p]})$.
Hence 
\[
\widetilde{\A}_{K,A}^{(0,r/p]}\otimes_{\widetilde{\A}_{K,A}^{(0,r]}}\widetilde{M}^{(0,r],(1)}=\widetilde{\A}_{K,A}^{(0,r/p]}\otimes_{\widetilde{\A}_{K,A}^{(0,r]}}\widetilde{M}^{(0,r],(2)},
\]
and it follows from condition $ii$ that $\varphi^{*}\widetilde{M}^{(0,r],(1)}=\varphi^{*}\widetilde{M}^{(0,r],(2)}.$
Since $\varphi:\widetilde{\A}_{K,A}^{(0,r]}\rightarrow\widetilde{\A}_{K,A}^{(0,r/p]}$
is an isomorphism, this gives $\widetilde{M}^{(0,r],(1)}=\widetilde{M}^{(0,r],(2)}$.

We now prove existence and uniqueness when $\widetilde{M}^{\dagger}$
is only assumed projective. To show existence, given $\widetilde{M}^{\dagger}$,
embed it as a direct summand of a free étale $\varphi$-module $\widetilde{F}^{\dagger}$,
as we may according to Lemma 3.9. and set $\widetilde{M}^{(0,r]}=\widetilde{M}^{\dagger}\cap\widetilde{F}^{(0,r]}$.
Let $\pi:\widetilde{F}^{\dagger}\rightarrow\widetilde{M}^{\dagger}$
be the projection. 

\textbf{Claim. }If $x\in\widetilde{F}^{(0,r]}$ then also $\pi(x)\in\widetilde{F}^{(0,r]}$. 

To see this, choose a basis $e_{1},...,e_{d}$ of $\widetilde{F}^{(0,r]}$.
Let $\Mat(\varphi)$ be the matrix of $\varphi$ with respect to this
basis. As the proof of the existence in the free case has shown, if
$r$ is taken to be sufficiently small, we can actually arrange that
$\Mat(\varphi)\in\GL_{d}(\widetilde{\A}_{K,A}^{(0,r]})$ (and not
just $\Mat(\varphi)\in\GL_{d}(\widetilde{\A}_{K,A}^{(0,r/p]})$).
The relation $\pi\circ\varphi=\varphi\circ\pi$ shows that if the
claim is true for sufficiently small $r$ it holds for any $r$, so
we may restrict to the case where $\Mat(\varphi)\in\GL_{d}(\widetilde{\A}_{K,A}^{(0,r]})$.
Now since $\widetilde{\A}_{K,A}^{\dagger}\otimes_{\widetilde{\A}_{K,A}^{(0,r]}}\widetilde{F}^{(0,r]}\xrightarrow{\sim}\widetilde{F}^{\dagger}$,
there is a matrix $\Mat(\pi)\in\M_{d}(\widetilde{\A}_{K,A}^{\dagger})$
representing $\pi$ with respect to the basis $e_{i}$. We need to
show that $\Mat(\pi)\in\M_{d}(\widetilde{\A}_{K,A}^{(0,r]})$. But
again using the relation $\pi\circ\varphi=\varphi\circ\pi$ we deduce
\[
\Mat(\varphi)\varphi(\Mat(\pi))=\Mat(\pi)\Mat(\varphi),
\]
which implies once more by Proposition 3.5 that $\Mat(\pi)\in\M_{d}(\widetilde{\A}_{K,A}^{(0,r]})$,
as required.

Now write $\widetilde{P}^{\dagger}$ for the étale $\varphi$-module
which is the complement of $\widetilde{M}^{\dagger}$ in $\widetilde{F}^{\dagger}$.
Letting $\widetilde{P}^{(0,r]}=\widetilde{P}^{\dagger}\cap\widetilde{F}^{(0,r]},$
the claim implies there is a direct decomposition
\[
\widetilde{F}^{(0,r]}=\widetilde{M}^{(0,r]}\oplus\widetilde{P}^{(0,r]},
\]
which respects the $\varphi$-action. With this given, the fact that
conditions $i$ and $ii$ hold for $\widetilde{F}^{(0,r]}$ implies
that they also hold for $\widetilde{M}^{(0,r]}$. The existence of
this decomposition also shows that $\widetilde{M}^{(0,r]}$ is projective.
This finishes the proof of existence of $\widetilde{M}^{(0,r]}$ in
general.

To prove uniquness, it suffices to show that if $\widetilde{M}^{(0,r],\prime}$
satisfies conditions $i$ and $ii$, and if $\widetilde{F}^{(0,r]}$
is constructed as above, and if $\widetilde{M}^{(0,r]}=\widetilde{F}^{(0,r]}\cap\widetilde{M}^{\dagger}$,
then $\widetilde{M}^{(0,r],\prime}=\widetilde{M}^{(0,r]}$.

But this follows from the uniqueness proved in the free case, because
\[
\widetilde{M}^{(0,r],\prime}\oplus\widetilde{P}^{(0,r]}=\widetilde{F}^{(0,r]}=\widetilde{M}^{(0,r]}\oplus\widetilde{P}^{(0,r]}.
\]
Applying the projection $\pi:\widetilde{F}^{\dagger}\rightarrow\widetilde{M}^{\dagger}$
we obtain $\widetilde{M}^{(0,r]}=\widetilde{M}^{(0,r],\prime}$. This
concludes the proof.
\end{proof}

\begin{prop}
Let $\widetilde{M}^{(0,r]}$ be a finite free $\widetilde{\A}_{K,A}^{(0,r]}$
semilinear representation of $\Gamma_{K}$. 

There exists an open normal subgroup $\Gamma_{L}=\Gal(K_{\infty}/L)$
of $\Gamma_{K}$ such that

1. There exists at most one free $\varphi^{-n}(\A_{K,A}^{(0,r/p^{n}]})$-submodule
$M_{n}^{(0,r]}$ of $\widetilde{M}^{(0,r]}$ such that 

i. The natural map $\widetilde{\A}_{K,A}^{(0,r]}\otimes_{\varphi^{-n}(\A_{K,A}^{(0,r/p^{n}]})}M_{n}^{(0,r]}\rightarrow\widetilde{M}^{(0,r]}$
is an isomorphism;

ii. $M_{n}^{(0,r]}$ is $\Gamma_{K}$-stable;

iii. $M_{n}^{(0,r]}$ has a $c_{3}$-fixed basis for the $\Gamma_{L}$-action.

2. If $n\geq n(\Gamma_{L},\widetilde{M}^{(0,r]})$ then $M_{n}^{(0,r]}$
exists. 
\end{prop}

\begin{proof}
We start by proving 2. Choose a basis of $\widetilde{M}^{(0,r]}$.
Then since $\widetilde{\A}_{K,A}^{(0,r],+}$ is open in $\widetilde{\A}_{K,A}^{(0,r]}$,
there exists an open subgroup $\Gamma_{L}$ of $\Gamma_{K}$ such
that $\Mat(g)\in\GL_{d}(\widetilde{\A}_{K,A}^{(0,r],+})$ for $g\in\Gamma_{L}$.
By possibly shrinking $\Gamma_{L}$, we may assume it to be normal,
and we may further assume for $g\in\Gamma_{L}$ we have $\val(\Mat(g)-1)>c_{3}$.
Let $\widetilde{M}^{(0,r],+}$ to be the $\widetilde{\A}_{K,A}^{(0,r],+}$-span
of this basis. It is a free $\widetilde{\A}_{K,A}^{(0,r],+}=\widetilde{\A}_{L,A}^{(0,r],+}$-semilinear\footnote{The equality $\widetilde{\A}_{K,A}^{(0,r],+}=\widetilde{\A}_{L,A}^{(0,r],+}$
occurs here because $K_{\infty}=L_{\infty}$.} representation of $\Gamma_{L}$, which satisfies the assumptions
in Theorem 4.11. Hence there exists a unique finite free $\varphi^{-n}(\A_{K,A}^{(0,r/p^{n}],+})$-submodule
$M_{n}^{(0,r],+}:=\D_{n}^{+}(\widetilde{M}^{(0,r],+})$ of $M^{(0,r],+}$,
satisfying that the natural map 
\[
\widetilde{\A}_{K,A}^{(0,r],+}\otimes_{\varphi^{-n}(\A_{K,A}^{(0,r/p^{n}],+})}M_{n}^{(0,r],+}\rightarrow\widetilde{M}^{(0,r],+}
\]
is an isomorphism, that $M_{n}^{(0,r],+}$ is $\Gamma_{L}$-stable,
and which has a $\varphi^{-n}(\A_{K,A}^{(0,r/p^{n}],+})$-basis which
is $c_{3}$-fixed for the action of $\Gamma_{L}$. 

Set $M_{n}^{(0,r]}=M_{n}^{(0,r],+}[1/T]$. In order to finish the
proof of existence of $M_{n}^{(0,r]}$, the only part which is not
yet clear is the following.

\textbf{Claim:} by possibly enlarging $n$, depending on $\widetilde{M}^{(0,r]}$,
we can arrange $M_{n}^{(0,r]}$ to be $\Gamma_{K}$-stable.

Indeed, choose coset representatives $\{g_{i}\}_{i\in I}$ for $\Gamma_{K}/\Gamma_{L}$.
For each such $g=g_{i}$, consider $g(M_{n}^{(0,r],+})$. If $e_{1},...,e_{d}$
is the $c_{3}$-fixed basis of $M_{n}^{(0,r],+}$ then $g(e_{1}),...,g(e_{d})$
is a basis of $g(M_{n}^{(0,r],+})$. It may not be $c_{3}$-fixed,
however. By continuity, we may find a nontrivial $\gamma\in\Gamma_{L}$,
with $\val(\Mat_{\{g(e_{i})\}}(\gamma)-1))>c_{3}$, and by taking
$n$ larger we can arrange that $n\geq n(\gamma)$. Lemma 4.7 then
implies that $g(M_{n}^{(0,r],+})=\D_{n}^{+}(g(\widetilde{M}^{(0,r],+}))$.
Choosing $n$ large enough for all the $g_{i}$ simultaneously, we
may arrange that $g(M_{n}^{(0,r],+})=\D_{n}^{+}(g(\widetilde{M}^{(0,r],+}))$
for every $g\in\Gamma_{K}$.

With this in mind, let $g\in\Gamma_{K}$. Then for some $t\in\Z$,
we have by continuity
\[
g(\widetilde{M}^{(0,r],+})\subset T^{-t}\widetilde{M}^{(0,r],+},
\]
so that 
\[
g(M_{n}^{(0,r],+})=\D_{n}^{+}(g(\widetilde{M}^{(0,r],+}))\subset\D_{n}^{+}(T^{-t}\widetilde{M}^{(0,r],+})=T^{-t}M_{n}^{(0,r],+}.
\]
Every element of $M_{n}^{(0,r]}$ can be written in the form $T^{t}m$
with $m\in M_{n}^{(0,r],+}$, and since
\[
g(T^{t}m)=[g(T^{t})/T^{t}]T^{t}g(m)\in M_{n}^{(0,r]},
\]
we see that $M_{n}^{(0,r]}$ is $\Gamma_{K}$-stable. This proves
the claim.

Finally, we show uniqueness. Suppose $M_{n}^{(0,r],(1)}$ and $M_{n}^{(0,r],(2)}$
are two submodules satisfying these properties. Let $M_{n}^{(0,r],(1),+}$
be the $\varphi^{-n}(\A_{K,A}^{(0,r/p^{n}],+})$-span of a $c_{3}$-fixed
basis in $M_{n}^{(0,r],(1)}$. Let $\widetilde{M}^{(0,r],(1),+}$
be the image of $\widetilde{\A}_{K,A}^{(0,r],+}\otimes_{\varphi^{-n}(\A_{K,A}^{(0,r/p^{n}],+})}M_{n}^{(0,r],(1)}$
in $\widetilde{M}^{(0,r]}$. Define $M_{n}^{(0,r],(2),+}$ and $\widetilde{M}^{(0,r],(2),+}$
similarly. Then we have for some sufficiently large $t$ the inclusions
\[
T^{t}\widetilde{M}^{(0,r],(2),+}\subset\widetilde{M}^{(0,r],(1),+}\subset T^{-t}\widetilde{M}^{(0,r],(2),+},
\]
which implies upon applying $\D_{n}^{+}$ that 
\[
T^{t}M_{n}^{(0,r],(2),+}\subset M_{n}^{(0,r],(1),+}\subset T^{-t}M_{n}^{(0,r],(2),+}.
\]
Hence $M_{n}^{(0,r],(1)}=M_{n}^{(0,r],(2)}$.
\end{proof}

\subsection{The equivalence of categories}

The following is an analogue of Lemma 3.9, with the action of $\varphi$
replaced by the $\Gamma_{K}$ action.
\begin{lem}
Let $R$ be a topological ring with a continuous action of $\Gamma_{K}$.
Let $M$ be a projective $R$-semilinear representation of $\Gamma_{K}$.
Then there exists a finite free $R$-semilinear representation of
$\Gamma_{K}$ which contains $M$ as a direct summand.
\end{lem}

\begin{proof}
Choose a topological generator $\gamma$ of $\Gamma_{K}$. Then by
the same argument proving Lemma 3.9, we may find a finite free $R$-module
$F$ endowed with an isomorphism $\gamma^{*}F\cong F$ and which contains
$M$ as a direct summand as a semilinear representation of $\gamma^{\Z}$.
Namely, to construct $F$, choose first a free $R$-module $G$ and
a projective $R$-module $P$ together with an $R$-module isomorphism
$M\oplus P\cong G$. Choose a basis $e_{1},...,e_{d}$ of $G$, and
give $G$ the structure of a semilinear representation of $\gamma^{\Z}$
by setting $\gamma(\sum r_{i}e_{i})=\sum\gamma(r_{i})e_{i}$. Then
$G\oplus P$ also admits such a structure, by taking the composite
\[
\gamma^{*}(G\oplus P)\cong\gamma^{*}G\oplus\gamma^{*}P\cong G\oplus\gamma^{*}P
\]
\[
=M\oplus P\oplus\gamma^{*}P\cong\gamma^{*}M\oplus P\oplus\gamma^{*}P\cong\gamma^{*}G\oplus P\cong G\oplus P.
\]
We then take $F:=(G\oplus P)\oplus M$. This is a free $R$-semilinear
representation of $\gamma^{\Z}$, with $M$ being a direct summand.

It remains to explain how to extend the action on $G\oplus P$ to
all of $\Gamma_{K}$. Let $\pi_{M}:G\rightarrow M,\pi_{P}:G\rightarrow P$
be the projections and given $\gamma^{k}\in\gamma^{\Z}$ let $\gamma_{\{e_{i}\}}^{k}:G\rightarrow G$
be the action obtained by fixing the basis $e_{1},...,e_{d}$. Unraveling
the definitions, one checks that the action of $\gamma^{k}$on elements
$(\sum x_{i}e_{i},p)\in G\oplus P$ is given by
\[
\gamma^{k}((\sum x_{i}e_{i},p))=(\gamma_{\{e_{i}\}}^{k}(\pi_{M}(\sum\gamma^{k}(x_{i})e_{i})+p),\pi_{P}(\sum\gamma^{k}(x_{i})e_{i})).
\]
Clearly, this formula can be extended to all elements of $\Gamma_{K}$,
as required.
\end{proof}
\begin{thm}
The functor $M^{\dagger}\mapsto\widetilde{M}^{\dagger}$ induces an
equivalence of categories from the category of projective étale $(\varphi,\Gamma_{K})$-modules
over $\A_{K,A}^{\dagger}$ to the category of projective étale $(\varphi,\Gamma_{K})$-modules
over $\widetilde{\A}_{K,A}^{\dagger}$. 
\end{thm}

\begin{proof}
We start by proving full faithfulness. As usual, we reduce to proving
that if $M^{\dagger}$ is a projective étale $(\varphi,\Gamma_{K})$-module
over $\widetilde{\A}_{K,A}^{\dagger}$ then $(M^{\dagger})^{\varphi,\Gamma_{K}}=(\widetilde{M}^{\dagger})^{\varphi,\Gamma_{K}}$.

The injectivity of $(M^{\dagger})^{\varphi,\Gamma_{K}}\rightarrow(\widetilde{M}^{\dagger})^{\varphi,\Gamma_{K}}$
is easy, the reason being that we already know that $\A_{K,A}^{\dagger}\rightarrow\widetilde{\A}_{K,A}^{\dagger}$
is injective, so it remains injective after tensoring with the projective,
hence flat, $\A_{K,A}^{\dagger}$-module $M^{\dagger}$. It then still
remains injective after taking fixed points.

For the surjectivity we argue as follows. Let $x\in(\widetilde{M}^{\dagger})^{\varphi,\Gamma_{K}}$.
Then $x\in(\widetilde{M}^{(0,r]})^{\varphi,\Gamma_{K}}$ for some
$r>0$. Take a finite free $\A_{K,A}^{\dagger}$-semilinear $\Gamma_{K}$-representation
$F^{\dagger}$ which contains $M^{\dagger}$ as a direct summand as
we may according to Lemma 5.4. Choose a basis $e_{1},...,e_{d}$ of
$F^{\dagger}$. We can write $\sum a_{i}e_{i}=x$ with $a_{i}\in\widetilde{\A}_{K,A}^{(0,r]}$.
Choose a nontorsion $\gamma\in\Gamma_{K}$. By possibly making $r$
smaller, we can arrange that $\Mat_{\{e_{i}\}}(\gamma)$ also has
coefficients in $\widetilde{\A}_{K,A}^{(0,r]}$. Since $x$ is fixed
by $\gamma$, we obtain the equation of $\widetilde{\A}_{K,A}^{(0,r]}$-valued
matrices
\[
\Mat_{\{e_{i}\}}(\gamma)\gamma(a)=a.
\]
where $a$ is the vector of the $a_{i}$. Replacing $\gamma$ by $\gamma^{p^{k}}$for
$k\gg0$ we may arrange in addition that $\val(\Mat_{\{e_{i}\}}(\gamma)-1)>c_{3}$.\textbf{
}So by Lemma 4.7 we know that for $n\gg0$, we have $a_{i}\in\varphi^{-n}(\A_{K,A}^{(0,r/p^{-n}]})$,
which is contained in $\varphi^{-n}(\A_{K,A}^{\dagger})$. Hence $x\in\varphi^{-n}(\A_{K,A}^{\dagger})\otimes_{\A_{K,A}^{\dagger}}M^{\dagger}$,
and since $x$ is fixed by $\varphi$, we see after $n$ successive
applications of $\varphi$ that $x\in M^{\dagger}$. This shows that
$x\in M^{\dagger}\cap(\widetilde{M}^{\dagger})^{\varphi,\Gamma_{K}}=(M^{\dagger})^{\varphi,\Gamma_{K}}$,
as required.

Next, we prove essential surjectivity. Let $\widetilde{M}^{\dagger}$
be a projective étale $\varphi$-module over $\widetilde{\A}_{K,A}^{\dagger}$.
Let $\widetilde{M}^{(0,r]}$ be as in Proposition 5.2. Then $\widetilde{M}^{(0,r]}$
is a projective $\widetilde{\A}_{K,A}^{(0,r]}$-semilinear representation
of $\Gamma_{K}$, so by Lemma 5.4, we may find a free $\widetilde{\A}_{K,A}^{(0,r]}$-semilinear
$\Gamma_{K}$-representation $\widetilde{F}^{(0,r]}$ and a projective
$\widetilde{\A}_{K,A}^{(0,r]}$-semilinear $\Gamma_{K}$-representation
$\widetilde{P}^{(0,r]}$ such that $\widetilde{M}^{(0,r]}\oplus\widetilde{P}^{(0,r]}=\widetilde{F}^{(0,r]}$.
By Proposition 5.3, we can find for $n\gg0$ a free $\varphi^{-n}(\A_{K,A}^{(0,r/p^{n}]})$-submodule
$F_{n}^{(0,r]}\subset\widetilde{F}^{(0,r]}$ which is $\Gamma_{K}$-stable.

\textbf{Claim. }Let $\pi:\widetilde{F}^{(0,r]}\rightarrow\widetilde{M}^{(0,r]}$
denote the projection. If $x\in F_{n}^{(0,r]}$ then $\pi(x)\in F_{n}^{(0,r]}$.

To see this, choose a basis $e_{1},...,e_{d}$ of $F_{n}^{(0,r]}$.
Choose $\gamma\in\Gamma_{K}$ nontorsion and let $\Mat(\gamma)$ be
the matrix of $\gamma$ with respect to this basis. Since $F_{n}^{(0,r]}$
spans $\widetilde{F}^{(0,r]}$ as an $\widetilde{\A}_{K,A}^{(0,r]}$-module,
there is a matrix $\Mat(\pi)\in\M_{d}(\widetilde{\A}_{K,A}^{(0,r]})$
representing $\pi$ with respect to the basis $e_{i}$. The relation
$\pi\circ\gamma=\gamma\circ\pi$ gives
\[
\Mat(\gamma)\gamma(\Mat(\pi))=\Mat(\pi)\Mat(\gamma),
\]
and after replacing $\gamma$ by $\gamma^{p^{k}}$ for $k\gg0$ we
may assume $\val(\Mat(\gamma)-1)>c_{3}$. This implies by Lemma 4.7
that $\Mat(\pi)\in\M_{d}(\varphi^{-n}(\A_{K,A}^{(0,r/p^{n}]}))$,
as required.

Set $M_{n}^{(0,r]}=F_{n}^{(0,r]}\cap\widetilde{M}^{(0,r]}$ and $P_{n}^{(0,r]}=F_{n}^{(0,r]}\cap\widetilde{P}^{(0,r]}$.
Then the claim shows that $M_{n}^{(0,r]}\oplus P_{n}^{(0,r]}=F_{n}^{(0,r]}$.
The isomorphism $\widetilde{\A}_{K,A}^{(0,r]}\otimes_{\varphi^{-n}(\A_{K,A}^{(0,r/p^{n}]})}F_{n}^{(0,r]}\xrightarrow{\sim}\widetilde{F}^{(0,r]}$
implies that $\widetilde{\A}_{K,A}^{(0,r]}\otimes_{\varphi^{-n}(\A_{K,A}^{(0,r/p^{n}]})}M_{n}^{(0,r]}\xrightarrow{\sim}\widetilde{M}^{(0,r]}$
and hence also
\[
\widetilde{\A}_{K,A}^{\dagger}\otimes_{\varphi^{-n}(\A_{K,A}^{(0,r/p^{n}]})}M_{n}^{(0,r]}\xrightarrow{\sim}\widetilde{M}^{\dagger}.
\]
It is also clear that $M_{n}^{(0,r]}$ is $\Gamma_{K}$-stable. We
set
\[
M^{\dagger}:=\A_{K,A}^{\dagger}\otimes_{\A_{K,A}^{(0,r/p^{n}]}}\varphi^{n}(M_{n}^{(0,r]}),
\]
then $M^{\dagger}$ is a $\Gamma_{K}$-stable, projective $\A_{K,A}^{\dagger}$-submodule
of $\widetilde{M}^{\dagger}$, and the natural map $\widetilde{\A}_{K,A}^{\dagger}\otimes_{\A_{K,A}^{\dagger}}M^{\dagger}\xrightarrow{\sim}\widetilde{M}^{\dagger}$
is an isomorphism. It remains to show that $M^{\dagger}$ is $\varphi$-stable
and an étale $\varphi$-module. To do this, simply notice that the
uniqueness of $F_{n}^{(0,r]}$ implies the uniqueness of $M_{n}^{(0,r]}$,
and so if $n$ is sufficiently large so that $M_{n-1}^{(0,r/p]}$
and $M_{n}^{(0,r]}$ are both defined, we get
\[
\varphi(M_{n}^{(0,r]})=M_{n-1}^{(0,r/p]}=\varphi^{-(n-1)}(\A_{K,A}^{(0,r/p^{n-1}]})\otimes_{\varphi^{-n}(\A_{K,A}^{(0,r/p^{n}]})}M_{n}^{(0,r]},
\]
which implies both that $M^{\dagger}$ is $\varphi$-stable and that
the action of $\varphi$ is invertible. This finishes the proof.
\end{proof}

\section{The main theorem}

In this section, we conclude with the proof of overconvergence of
$(\varphi,\Gamma_{K})$-modules over $\A_{K,A}$.
\begin{lem}
The functor $M\mapsto\widetilde{M}:=\widetilde{\A}_{K,A}\otimes_{\A_{K,A}}M$
from projective étale $(\varphi,\Gamma_{K})$-modules over $\A_{K,A}$
to projective étale $(\varphi,\Gamma_{K})$-modules over $\widetilde{\A}_{K,A}$
is fully faithful.
\end{lem}

\begin{proof}
As usual, using Lemma 3.10 we can reduce to checking that the natural
map $M^{\varphi=1}\xrightarrow{\sim}\widetilde{M}^{\varphi=1}$ is
an isomorphism. Since $\A_{K,A}$ and $\widetilde{\A}_{K,A}$ are
$p$-adically complete, and since $M$ is free, we have compatible
isomorphisms $M\cong\varprojlim_{a}M\otimes_{\A_{K,A}}\A_{K,A/p^{a}}$
and $\widetilde{M}\cong\varprojlim_{n}\widetilde{M}\otimes_{\widetilde{\A}_{K,A}}\widetilde{\A}_{K,A/p^{a}}$.
Since $\varphi$-invariants are compatible with inverse limits, we
are reduced to the case where $p^{a}A=0$. But in this case the statement
is known, by \cite[Prop. 2.6.6]{EG19}.
\end{proof}
Finally, we can prove the main theorem.
\begin{thm}
The functor $M^{\dagger}\mapsto M:=\A_{K,A}\otimes_{\A_{K,A}^{\dagger}}M^{\dagger}$
induces an equivalence of categories from the category of projective
étale $(\varphi,\Gamma_{K})$-modules over $\A_{K,A}^{\dagger}$ to
the category of projective étale $(\varphi,\Gamma_{K})$-modules over
$\A_{K,A}$.
\end{thm}

\begin{proof}
This follows by combining Theorem 3.13, Theorem 5.5 and Lemma 6.1.
\end{proof}

\section{Appendix: coefficient rings}

In this section we establish the properties of the coefficient rings
claimed in $\mathsection2$ without proof (recall $\mathsection2.1$
for the definitions of the rings). The most difficult part is proving
the injectivity and continuity of maps into $\widetilde{\A}_{A}$
(Theorem 7.9 and its consequences). This is because of nonflatness
combined with torsion at $p$ and $[\varpi]$, as well as the rings
involved being non-noetherian.

\subsection{Compatibility with $H$-invariants}
\begin{prop}
We have natural isomorphisms

i. $\widetilde{\A}_{K,A}^{+}\cong W(\mathcal{O}_{\widehat{K}_{\infty}}^{\flat})_{A}.$

ii. $\widetilde{\A}_{K,A}\cong W(\widehat{K}_{\infty}^{\flat})_{A}$.

iii. $\widetilde{\A}_{K,A}^{(0,r]}\cong(\widetilde{\A}_{K}^{+}\left\langle p/[\varpi]^{1/r}\right\rangle \widehat{\otimes}A)[1/[\varpi]]$.

iv. $\widetilde{\A}_{K,A}^{(0,\infty)}\cong W(\mathcal{O}_{\widehat{K}_{\infty}}^{\flat})_{A}\left[1/[\varpi]\right]$.

v. $\widetilde{\A}_{K,A}^{\dagger}=\varinjlim_{r}(\widetilde{\A}_{K}^{+}\left\langle p/[\varpi]^{1/r}\right\rangle \widehat{\otimes}A)[1/[\varpi]]$.

Here, in $iii$ and $v$, the tensor products are completed with respect
to the $[\varpi]$-adic topology.
\end{prop}

\begin{proof}
It is clear that $W(\mathcal{O}_{\widehat{K}_{\infty}}^{\flat})=W(\mathcal{O}_{\C}^{\flat})^{H_{K}}$
and that $[\varpi]$ is fixed by $H_{K}$. Taking fixed points commutes
with inverse limits so this proves parts $i$ and $ii$. In addition,
$i$ implies $iv$ and $iii$ implies $v$. It remains to prove $iii$.
For ease of notation, we write $H$ for $H_{K}$. We prove this by
showing two claims.

\textbf{Claim 1. }The natural map $\widetilde{\A}^{(0,r],\circ,H}\widehat{\otimes}A:=(\widetilde{\A}^{(0,r],\circ,H}\otimes A)_{[\varpi]}^{\wedge}\rightarrow(\widetilde{\A}_{A}^{(0,r],\circ})^{H}$
is injective.

To prove this, we argue in steps.

\textbf{Step 1. }The natural maps $\widetilde{\A}^{(0,r],\circ,H}/p\rightarrow\widetilde{\A}^{(0,r],\circ}/p$
and $\mathcal{O}_{\C}^{\flat,H}[X]/X\varpi^{1/r}\rightarrow\mathcal{O}_{\C}^{\flat}[X]/X\varpi^{1/r}$
are injective. 

Indeed, the cokernel of $\widetilde{\A}^{(0,r],\circ,H}\rightarrow\widetilde{\A}^{(0,r],\circ}$
(resp of $\mathcal{O}_{\C}^{\flat,H}[X]\rightarrow\mathcal{O}_{\C}^{\flat}[X]$)
is $p$-torsionfree (resp is $X\varpi^{1/r}$-torsionfree).

\textbf{Step 2.} The natural isomorphism $\widetilde{\A}^{(0,r],\circ}/p\xrightarrow{\sim}\mathcal{O}_{\C}^{\flat}[X]/X\varpi^{1/r}$
induces a natural isomorphism $\widetilde{\A}^{(0,r],\circ,H}/p\xrightarrow{\sim}\mathcal{O}_{\C}^{\flat,H}[X]/X\varpi^{1/r}$. 

For this, consider the commutative diagram
\[
\xymatrix{\widetilde{\A}^{(0,r],\circ,H}/p\ar[r]\ar[d] & \mathcal{O}_{\C}^{\flat,H}[X]/X\varpi^{1/r}\ar[d]\\
\widetilde{\A}^{(0,r],\circ}/p\ar[r] & \mathcal{O}_{\C}^{\flat}[X]/X\varpi^{1/r}
}
.
\]
The bottom horizontal map is an isomorphism, hence injective. The
two vertical maps are also injective, by step 1. Hence $\widetilde{\A}^{(0,r],\circ,H}/p\rightarrow\mathcal{O}_{\C}^{\flat,H}[X]/X\varpi^{1/r}$
is injective. For surjectivity, it suffices to show that $\A_{\inf}^{H}\rightarrow\mathcal{O}_{\C}^{\flat,H}$
is surjective, which is clear, for example because the Teichmuller
map gives a section.

\textbf{Step 3.} Set $S=\text{\ensuremath{\widetilde{\A}^{(0,r],\circ}}}/\widetilde{\A}^{(0,r],\circ,H}$.
Then $S$ is $p$-torsionfree.

\textbf{Step 4.} If $A$ is killed by $p$, we claim that the $\varpi^{\infty}$-torsion
in $K\otimes A$ is $\varpi^{1/r}$-torsion. 

Indeed, choose an isomorphism $A\cong\bigoplus_{i\in I}\F_{p}$. By
step 3, we have an exact sequence
\[
0\rightarrow\widetilde{\A}^{(0,r],\circ,H}\otimes A\rightarrow\widetilde{\A}^{(0,r],\circ}\otimes A\rightarrow S\otimes A\rightarrow0,
\]
which by step 2 is isomorphic to
\[
0\rightarrow\bigoplus_{i\in I}\mathcal{O}_{\C}^{\flat,H}[X]/X\varpi^{1/r}\rightarrow\bigoplus_{i\in I}\mathcal{O}_{\C}^{\flat}[X]/X\varpi^{1/r}\rightarrow\bigoplus_{i\in I}\mathcal{O}_{\C}^{\flat}[X]/(\mathcal{O}_{\C}^{\flat,H}[X],X\varpi^{1/r})\rightarrow0,
\]
so this is clear.

\textbf{Step 5.} If $A$ is killed by $p^{N}$, we claim that the
$[\varpi]^{\infty}$-torsion in $K\otimes A$ is killed by $[\varpi]^{N/r}$. 

To see this, step 3 to obtain an exact sequence
\[
0\rightarrow S\otimes pA\rightarrow S\otimes A\rightarrow S\otimes A/p\rightarrow0
\]
which implies the statement by an obvious devissage, using step 4.

\textbf{Step 6.} If $A$ is $p$-torsionfree, we claim that $S\otimes A$
is $[\varpi]$-torsionfree. 

This follows from flat base change, because
\[
\mathrm{Tor}_{1}^{\A_{\inf}}(\A_{\inf}/[\varpi],K)\cong\mathrm{Tor}_{1}^{\A_{\inf}\otimes A}((\A_{\inf}\otimes A)/[\varpi],S\otimes A).
\]
As $[\varpi]$ is a nonzero divisor in $\A_{\inf}\otimes A$ and $S$
is $[\varpi]$-torsionfree, the claim follows.

\textbf{Step 7.} Since $A$ is noetherian, we have $A[p^{\infty}]=A[p^{N}]$
for $N\gg0$. Since $A/A[p^{N}]$ is $p$-torsionfree, we have an
exact sequence
\[
0\rightarrow S\otimes A[p^{N}]\rightarrow S\otimes A\rightarrow S\otimes A/A[p^{N}]\rightarrow0,
\]
which, combining steps 5 and 6, shows that $S\otimes A$ is bounded
$[\varpi]$-torsion.

\textbf{Step 8.} Finally, we have an exact sequence
\[
0\rightarrow\widetilde{\A}^{(0,r],\circ,H}\otimes A\rightarrow\widetilde{\A}^{(0,r],\circ}\otimes A\rightarrow S\otimes A\rightarrow0.
\]
Since $K\otimes A$ has bounded $[\varpi]$-torsion, it follows from
Lemma 7.2 below that the sequence remains exact after $[\varpi]$-adic
completion. Hence the map $\widetilde{\A}^{(0,r],\circ,H}\widehat{\otimes}A\rightarrow\widetilde{\A}_{A}^{(0,r],\circ}$
is injective, which concludes the proof of claim 1.

\textbf{Claim 2. }The map $\widetilde{\A}^{(0,r],\circ,H}\widehat{\otimes}A\rightarrow(\widetilde{\A}_{A}^{(0,r],\circ})^{H}$
is almost surjective.

Since both sides are $[\varpi]$-adically complete, is enough to prove
almost surjectivity mod $[\varpi]^{1/r}$. We have a natural isomorphism
\[
\widetilde{\A}_{A}^{(0,r],\circ}/[\varpi]^{1/r}\cong\mathcal{O}_{\C}^{\flat}/\varpi^{1/r}[X]\otimes_{\F_{p}}A/p.
\]
Since $A/p$ is an $\F_{p}$-vector space, it is free over $\F_{p}$,
and so
\[
(\widetilde{\A}_{A}^{(0,r],\circ}/[\varpi]^{1/r})^{H}=(\mathcal{O}_{\C}^{\flat}/\varpi^{1/r}[X])^{H}\otimes_{\F_{p}}A/p.
\]
As $\widetilde{\A}^{(0,r],\circ,H}\widehat{\otimes}A$ surjects onto
$\mathcal{O}_{\C}^{\flat,H}/\varpi^{1/r}[X]\otimes_{\F_{p}}A/p$ (as
follows from step 2 in the proof of the previous claim), it suffices
to prove that $\mathcal{O}_{\C}^{\flat,H}/\varpi^{1/r}$ almost surjects
onto $(\mathcal{O}_{\C}^{\flat}/\varpi^{1/r})^{H}$. This follows
from $\H^{1}(H,\varpi\mathcal{O}_{\C}^{\flat})$ being almost zero,
which is \cite[Lem. IV.2.3]{Co98}. 

Combining both of the claims, we deduce that there is a natural isomorphism
$(\widetilde{\A}_{A}^{(0,r]})^{H}\cong\widetilde{\A}^{(0,r],H}\widehat{\otimes}A$.
This is simply different notation for $\widetilde{\A}_{K,A}^{(0,r]}\cong(\widetilde{\A}_{K}^{+}\left\langle p/[\varpi]^{1/r}\right\rangle \widehat{\otimes}A)[1/[\varpi]]$,
and so proves part $iii$ of the proposition.
\end{proof}

\subsection{Torsion and completions}

The main goal of this subsection is to construct a natural continuous
map $\widetilde{\A}_{A}^{(0,r]}\rightarrow\widetilde{\A}_{A}$ and
to prove it is injective. This will allow us to prove that the direct
limits defining $\widetilde{\A}_{K,A}^{\dagger}$ and $\A_{K,A}^{\dagger}$
have injective transition maps. 

To establish basic properties of the modules and rings introduced
above, it will be necessary to prove that certain completion operations
are well behaved. Our rings will usually be nonnoetherian, so such
results are not automatic in our setting. However, it will turn out
that in the situations we consider here the torsion appearing is bounded.
Fortunately, this weaker finiteness condition will suffice for controlling
the completions appearing in this article by virtue of the following
simple lemma.
\begin{lem}
Let $R$ be a ring, $x$ a nonzerodivisor of $R$, and $0\rightarrow M\rightarrow N\rightarrow P\rightarrow0$
a short exact sequence of $R$-modules.

If $P$ has bounded $x$-torsion then the $x$-adic completion $0\rightarrow M_{x}^{\wedge}\rightarrow N_{x}^{\wedge}\rightarrow P_{x}^{\wedge}\rightarrow0$
is exact.
\end{lem}

\begin{proof}
The exactness on the right is automatic by Nakayama's lemma.\footnote{This does not require the modules to be assumed finitely generated,
see \cite[Tag 0315 (2)]{Sta}}

We now turn to explain the exactness elsewhere. Firstly, we note that
since $x$ is a nonzerodivisor, $x^{n}$ is also a nonzerodivisor.
Replacing $x$ by its own power, we may and do assume $P[x]=P[x^{\infty}]$.
Next, we observe that $\mathrm{Tor}_{1}(R/x^{n},P)=P[x^{n}]$, and
the map $\mathrm{Tor}_{1}(R/x^{n},P)\rightarrow\mathrm{Tor}_{1}(R/x^{n-1},P)$
induced from $R/x^{n}\rightarrow R/x^{n-1}$ corresponds to the map
$P[x^{n}]\xrightarrow{x}P[x^{n-1}]$.

With this given, we have exact sequences
\[
P[x^{n}]\rightarrow M/x^{n}\rightarrow N/x^{n}\rightarrow P/x^{n}\rightarrow0.
\]
Let $K_{n}$ be the image of $P[x^{n}]$ in $M/x^{n}$. Then we have
short exact sequences
\[
0\rightarrow(M/x^{n})/K_{n}\rightarrow N/x^{n}\rightarrow P/x^{n}\rightarrow0,
\]
and taking the inverse limit we obtain an exact sequence
\[
0\rightarrow\varprojlim_{n}(M/x^{n})/K_{n}\rightarrow N_{x}^{\wedge}\rightarrow P_{x}^{\wedge}.
\]
It remains to show that $\varprojlim_{n}(M/x^{n})\rightarrow\varprojlim_{n}(M/x^{n})/K_{n}$
is an isomorphism. To show this, it suffices to show that $\varprojlim_{n}K_{n}$
and $\R^{1}\varprojlim_{n}K_{n}$ both vanish. Since $P[x^{\infty}]=P[x]$,
the transition maps $f_{n+1}:K_{n+1}\rightarrow K_{n}$, which are
induced from the multiplication by $x$ from $P[x^{n+1}]$ to $P[x^{n}]$,
are all $0$. This implies that the complex
\[
\prod_{n\geq1}K_{n}\rightarrow\prod_{n\geq1}K_{n}
\]
\[
(x_{n})\mapsto(x_{n}-f_{n+1}(x_{n+1}))
\]
is exact. This complex computes $\R^{i}\varprojlim_{n}K_{n}$, so
we are done.
\end{proof}
The next proposition allows us to control torsion.
\begin{prop}
Let $v$ be an element of the maximal ideal of $\mathcal{O}_{\C}^{\flat}$.

i. $A$ has bounded $p$-torsion.

ii. $\A_{\inf}/[v]$ is $p$-torsionfree.

iii. $\A_{\inf}\otimes_{\Z_{p}}A$ is $[v]$-torsionfree.

iv. $\widetilde{\A}^{(0,r],\circ}\otimes_{\Z_{p}}A$ has bounded $[v]$-torsion.

v. $\widetilde{\A}_{A}^{(0,r],\circ}$ has bounded $[v]$-torsion.
\end{prop}

\begin{proof}
$i.$ The $\Z_{p}$-module $A$ is noetherian, since it is topologically
of finite type as a $\Z_{p}$-module. It therefore has bounded $p$-torsion. 

$ii.$ Let $x\in\A_{\inf}$ and suppose that $px\in[v]\A_{\inf}$.
If we write $x=\sum_{i\geq0}[x_{i}]p^{i}$ for the Teichmuller expansion
then $px=\sum_{i\geq0}[x_{i}]p^{i+1}\in[v]\A_{\inf}$, which implies
by uniqueness of the expansion that $x_{i}\in v\mathcal{O}_{\C}^{\flat}$,
hence $x$ itself is divisible by $[v]$.

$iii.$ The ring $\A_{\inf}$ is $[v]$-torsionfree and $\Z_{p}$-flat.
Tensoring the exact sequence 
\[
0\rightarrow\A_{\inf}\xrightarrow{[v]}\A_{\inf}\rightarrow\A_{\inf}/[v]\rightarrow0
\]
with $A$, we see that the $[v]$-torsion in $\A_{\inf}\otimes_{\Z_{p}}A$
is isomorphic to $\mathrm{Tor}_{1}^{\Z_{p}}(A,\A_{\inf}/[v])$. This
vanishes by $ii$.

$iv.$ The ring $\widetilde{\A}^{(0,r],\circ}$ ihas no $p$ or $[v]$-torsion.
This is because it is a subring of $\widetilde{\A}=W(\C^{\flat})$,
which has these properties. Tensoring the exact sequence 
\[
0\rightarrow\widetilde{\A}^{(0,r],\circ}\xrightarrow{[v]}\widetilde{\A}^{(0,r],\circ}\rightarrow\widetilde{\A}^{(0,r],\circ}/[v]\rightarrow0
\]
with $A$ shows that there is a natural isomorphism between the $[v]$-torsion
in $\widetilde{\A}^{(0,r],\circ}\otimes_{\Z_{p}}A$ and $\mathrm{Tor}_{1}^{\Z_{p}}(\widetilde{\A}^{(0,r],\circ}/[v],A)$.
Then for $N\gg0$ we have
\[
\mathrm{Tor}_{1}^{\Z_{p}}(\widetilde{\A}^{(0,r],\circ}/[v],A)\cong\mathrm{Tor}_{1}^{\Z_{p}}(\widetilde{\A}^{(0,r],\circ}/[v],A[p^{\infty}])=\mathrm{Tor}_{1}^{\Z_{p}}(\widetilde{\A}^{(0,r],\circ}/[v],A[p^{N}]),
\]
so we deduce that the $[v]$-torsion in $\widetilde{\A}^{(0,r],\circ}\otimes_{\Z_{p}}A$
is isomorphic to the $[v]$-torsion in $\widetilde{\A}^{(0,r],\circ}\otimes_{\Z_{p}}A[p^{N}]$. 

Replacing $A$ by $A[p^{N}]$, we may assume $p^{N}A=0$. We now prove
the $[v]$-torsion is bounded by induction on $N$. We may assume
$v=\varpi$. When $N=1$, we have $pA=0$, so $A$ is an $\F_{p}$-vector
space. Upon choosing a basis and using the isomorphism
\[
\widetilde{\A}^{(0,r],\circ}/p\cong\mathcal{O}_{\C}^{\flat}[X]/(X\varpi^{1/r}),
\]
we see that the $\varpi^{\infty}$-torsion in $\widetilde{\A}^{(0,r],\circ}\otimes_{\Z_{p}}A$
is killed by $\varpi^{1/r}$.

For general $N$, we see that the $\varpi^{\infty}$-torsion in $\widetilde{\A}^{(0,r],\circ}\otimes_{\Z_{p}}A$
is killed by $\varpi^{N/r}$ by devissage through use of the exact
sequence
\[
0\rightarrow(\widetilde{\A}^{(0,r],\circ}\otimes_{\Z_{p}}A[p])\rightarrow(\widetilde{\A}^{(0,r],\circ}\otimes_{\Z_{p}}A)\xrightarrow{p}(\widetilde{\A}^{(0,r],\circ}\otimes_{\Z_{p}}pA)\rightarrow0.
\]

$v.$ As in iv there is $N\gg0$ such that $A[p^{N}]=A[p^{\infty}]$,
and we may assume that $v=\varpi$. Consider $M\geq N/r$. We have
a commutative diagram
\[
\xymatrix{0\ar[r] & \mathrm{Tor}_{1}^{\Z_{p}}(\widetilde{\A}^{(0,r],\circ}/[\varpi]^{N/r},A)\ar[d]\ar[r] & \widetilde{\A}^{(0,r],\circ}\otimes_{\Z_{p}}A\ar[r]^{[\varpi]^{N/r}}\ar[d] & [\varpi]^{N/r}\widetilde{\A}^{(0,r],\circ}\otimes_{\Z_{p}}A\ar[d]^{[\varpi]^{M-N/r}}\ar[r] & 0\\
0\ar[r] & \mathrm{Tor}_{1}^{\Z_{p}}(\widetilde{\A}^{(0,r],\circ}/[\varpi]^{M},A)\ar[r] & \widetilde{\A}^{(0,r],\circ}\otimes_{\Z_{p}}A\ar[r]^{[\varpi]^{M}} & [\varpi]^{M}\widetilde{\A}^{(0,r],\circ}\otimes_{\Z_{p}}A\ar[r] & 0
}
.
\]
We claim that the leftmost map is an equality. To see this, notice
that by the argument proving $iv$, we have natural isomorphisms
\[
\mathrm{Tor}_{1}^{\Z_{p}}(\widetilde{\A}^{(0,r],\circ}/[\varpi]^{N/r},A)\cong\mathrm{Tor}_{1}^{\Z_{p}}(\widetilde{\A}^{(0,r],\circ}/[\varpi]^{N/r},A[p^{N}])
\]
and similarly 
\[
\mathrm{Tor}_{1}^{\Z_{p}}(\widetilde{\A}^{(0,r],\circ}/[\varpi]^{M},A)\cong\mathrm{Tor}_{1}^{\Z_{p}}(\widetilde{\A}^{(0,r],\circ}/[\varpi]^{M},A[p^{N}])
\]
and we know the two respective right hand sides are equal by what
was proven in $iv$. 

It follows from the snake lemma that the commutative diagram gives
an isomorphism between the two rows. Now again by iv, the $[\varpi]^{\infty}$-torsion
in $\widetilde{\A}^{(0,r],\circ}\otimes_{\Z_{p}}A$ is bounded. So
by Lemma 7.2, taking the $[\varpi]$-completion of both rows is exact.
In addition, by \cite[05GG]{Sta}, the $[\varpi]$-adic completion
of $[\varpi]^{k}\widetilde{\A}^{(0,r],\circ}\otimes_{\Z_{p}}A$ for
any $k$ is equal to $[\varpi]^{k}\widetilde{\A}_{A}^{(0,r],\circ}$.
We deduce that the $[\varpi]^{N/r}$-torsion in $\widetilde{\A}_{A}^{(0,r],\circ}$
is equal to the $[\varpi]^{M}$ -torsion in $\widetilde{\A}_{A}^{(0,r],\circ}$
for $M\geq N/r$. This proves the $[\varpi]^{\infty}$-torsion is
bounded in $\widetilde{\A}_{A}^{(0,r],\circ}$.
\end{proof}
We record the following result that will be used later in $\mathsection5$.
\begin{cor}
For $1/r\in\Z[1/p]_{>0}$ the topology on $\widetilde{\A}_{A}^{(0,r]}$
is defined by the valuation given by
\[
\val^{(0,r]}(x)=(p/p-1)\sup\{t\in\Z[1/p]:x\in[\varpi]^{t}\widetilde{\A}_{A}^{(0,r],+}\}.
\]
\end{cor}

\begin{proof}
It follows from the definitions that $\widetilde{\A}_{A}^{(0,r]}$
has $\widetilde{\A}_{A}^{(0,r],+}$ as an open subring, for which
the topology is $[\varpi]$-adic. The only thing left to check is
that $\widetilde{\A}_{A}^{(0,r],+}$ is $[\varpi]$-adically separated,
so that $\val^{(0,r]}$ defines a valuation. But for $N\gg0$ we have
by Proposition $7.3.v$ that $\widetilde{\A}_{A}^{(0,r],\circ}[[\varpi]^{\infty}]=\widetilde{\A}_{A}^{(0,r],\circ}[[\varpi]^{N}]$,
so that 
\[
\widetilde{\A}_{A}^{(0,r],+}=\widetilde{\A}_{A}^{(0,r],\circ}/\widetilde{\A}_{A}^{(0,r],\circ}[[\varpi]^{N}],
\]
with $\widetilde{\A}_{A}^{(0,r],\circ}[[\varpi]^{N}]$ closed. Since
$\widetilde{\A}_{A}^{(0,r],\circ}$ is $[\varpi]$-adically separated,
the corollary follows.
\end{proof}
Next, we define two $p$-adically completed $\Z_{p}$-modules
\[
\A_{\inf,A}:=\varprojlim_{a}(\A_{\inf}\otimes_{\Z_{p}}A)/p^{a},
\]
\[
\A_{\inf,A}\left\langle p/[\varpi]^{1/r}\right\rangle :=\varprojlim_{a}(\A_{\inf}\otimes_{\Z_{p}}A)\left[p/[\varpi]^{1/r}\right]/p^{a},
\]
which are rings if $A$ is. Clearly, there is a map $\A_{\inf,A}\rightarrow\A_{\inf,A}\left\langle p/[\varpi]^{1/r}\right\rangle $
which is continuous with respect to the $p$-adic topology.

These two will play an auxiliary role in what follows. We shall need
these as it will be easier to construct maps out of them, and then
later extend these to maps to the objects we are concerned with. More
precisely, note that $\A_{\inf,A}$ is quite close to $\widetilde{\A}_{A}^{(0,\infty],\circ}=W(\mathcal{O}_{F}^{\flat})_{A}$
while $\A_{\inf,A}\left\langle p/[\varpi]^{1/r}\right\rangle $ is
close to being equal to $\widetilde{\A}_{A}^{(0,r],\circ}$, with
these latter rings being those of true importance. The subtle differences
in the two pairs occur because of the distinction between the $[\varpi]$-adic,
$p$-adic and $(p,[\varpi])$-adic completions. 
\begin{lem}
$[\varpi]$-adic completion induces an isomorphism $\A_{\inf,A}\left\langle p/[\varpi]^{1/r}\right\rangle _{[\varpi]}^{\wedge}\cong\widetilde{\A}_{A}^{(0,r],\circ}$.
\end{lem}

\begin{proof}
Recall that $\widetilde{\A}_{A}^{(0,r],\circ}$ is defined as the
$[\varpi]$-adic completion of $\A_{\inf}\left\langle p/[\varpi]^{1/r}\right\rangle \otimes_{\Z_{p}}A$.
We have

\[
\begin{aligned}\begin{aligned}\A_{\inf,A}\left\langle \frac{p}{[\varpi]^{1/r}}\right\rangle _{[\varpi]}^{\wedge} & =\left[(\A_{\inf}\otimes_{\Z_{p}}A)\left[\frac{p}{[\varpi]^{1/r}}\right]_{p}^{\wedge}\right]_{[\varpi]}^{\wedge}\\
 & \cong(\A_{\inf}\otimes_{\Z_{p}}A)\left[\frac{p}{[\varpi]^{1/r}}\right]_{[\varpi]}^{\wedge}.
\end{aligned}
\end{aligned}
\]
Further,
\[
\begin{aligned}\begin{aligned}(\A_{\inf}\otimes_{\Z_{p}}A)\left[\frac{p}{[\varpi]^{1/r}}\right]/[\varpi]^{a} & \cong(\A_{\inf}\otimes_{\Z_{p}}A)[X]/(X[\varpi]^{1/r}-p)/[\varpi]^{a}\\
 & \cong(\A_{\inf}[X]/(X[\varpi]^{1/r}-p)\otimes_{\Z_{p}}A)/[\varpi]^{a}\\
 & \cong(\A_{\inf}\left[\frac{p}{[\varpi]^{1/r}}\right]\otimes_{\Z_{p}}A)/[\varpi]^{a}.
\end{aligned}
\end{aligned}
\]
Taking the limit over $a$ we obtain the desired isomorphism. 
\end{proof}
\begin{lem}
The natural map $\A_{\inf,A}\rightarrow\A_{\inf,A}\left\langle p/[\varpi]^{1/r}\right\rangle $
is injective.
\end{lem}

\begin{proof}
We start by showing that $\A_{\inf}\otimes_{\Z_{p}}A\rightarrow(\A_{\inf}\otimes_{\Z_{p}}A)[p/[\varpi]^{1/r}]$
is injective. It suffices to show that
\[
(\A_{\inf}\otimes_{\Z_{p}}A)\cap(X[\varpi]^{1/r}-p)(\A_{\inf}\otimes_{\Z_{p}}A)[X]=0,
\]
where the intersection is taken in $(\A_{\inf}\otimes_{\Z_{p}}A)[X]$.
Indeed, suppose $f=f(X)$ is in the intersection, then we may write
\[
f(X)=(X[\varpi]^{1/r}-p)g(X)
\]
with $g(X)=a_{0}+...+a_{d}X^{d}\in(\A_{\inf}\otimes_{\Z_{p}}A)[X]$
and $d\geq0$, with $a_{d}\neq0$ unless $g(X)=0$. Since $f\in\A_{\inf}\otimes_{\Z_{p}}A$,
the coefficient of $X^{d+1}$ in $f(X)$ is $0$, which gives $a_{d}[\varpi]^{1/r}=0$
in $\A_{\inf}\otimes_{\Z_{p}}A$. By $7.3.iii$, the ring $\A_{\inf}\otimes_{\Z_{p}}A$
is $[\varpi]^{1/r}$-torsionfree so we must have $a_{d}=0$ which
means $g(X)=0$, and hence $f(X)=0$.

Now if $p^{a}A=0$ the proposition holds because what we have just
shown, since in this case $\A_{\inf}\otimes_{\Z_{p}}A=\A_{\inf,A}$
and $(\A_{\inf}\otimes_{\Z_{p}}A)[p/[\varpi]^{1/r}]=\A_{\inf,A}\left\langle p/[\varpi]^{1/r}\right\rangle $.
In general, the map $\A_{\inf,A}\rightarrow\A_{\inf,A}\left\langle p/[\varpi]^{1/r}\right\rangle $
is obtained by taking the inverse limit over $a$ of the injective
maps $\A_{\inf,A/p^{a}}\rightarrow\A_{\inf,A/p^{a}}\left\langle p/[\varpi]^{1/r}\right\rangle $,
so it is injective. This concludes the proof.
\end{proof}

We may now construct a natural map $\widetilde{\A}_{A}^{(0,r]}\rightarrow\widetilde{\A}_{A}$
as follows. First, we have a natural map $\A_{\inf}\otimes_{\Z_{p}}A\rightarrow W_{a}(\mathcal{O}_{\C}^{\flat})_{A}\left[\frac{1}{[\varpi]}\right]$,
defined as the composition
\[
\A_{\inf}\otimes_{\Z_{p}}A=W(\mathcal{O}_{\C}^{\flat})\otimes_{\Z_{p}}A\rightarrow W_{a}(\mathcal{O}_{\C}^{\flat})\otimes_{\Z_{p}}A\rightarrow W_{a}(\mathcal{O}_{\C}^{\flat})_{A}\rightarrow W_{a}(\mathcal{O}_{\C}^{\flat})_{A}\left[\frac{1}{[\varpi]}\right].
\]
This induces a map
\[
(\A_{\inf}\otimes_{\Z_{p}}A)\left[\frac{p}{[\varpi]^{1/r}}\right]/p^{a}\rightarrow W_{a}(\mathcal{O}_{\C}^{\flat})_{A}\left[\frac{1}{[\varpi]}\right],
\]
and taking limits as $a\rightarrow\infty$, we get a map $\A_{\inf,A}\left\langle p/[\varpi]^{1/r}\right\rangle \rightarrow\widetilde{\A}_{A}$,
which is by construction continuous for the $p$-adic topology on
$\A_{\inf,A}\left\langle p/[\varpi]^{1/r}\right\rangle $ and the
natural topology on $\widetilde{\A}_{A}$. Recall this latter topology
is the inverse limit topology induced from $\widetilde{\A}_{A}=\varprojlim_{a}W_{a}(\mathcal{O}_{\C}^{\flat})_{A}\left[1/[\varpi]\right]$,
for which a basis of open neighborhoods of $0$ in $\widetilde{\A}_{A}$
is given by $\{[\varpi]^{k/r}W(\mathcal{O}_{\C}^{\flat})_{A}+p^{k}W(\C^{\flat})_{A}\}_{k\geq1}$.
The construction of the map $\widetilde{\A}_{A}^{(0,r]}\rightarrow\widetilde{\A}_{A}$
is concluded by the Lemma 7.5 and the following lemma.
\begin{lem}
The map $\A_{\inf,A}\left\langle p/[\varpi]^{1/r}\right\rangle \rightarrow\widetilde{\A}_{A}$
is continuous for the natural topology on $\widetilde{\A}_{A}$ and
the $[\varpi]$-adic topology on $\A_{\inf,A}\left\langle p/[\varpi]^{1/r}\right\rangle $.
\end{lem}

\begin{proof}
A basis of open neighborhoods of $0$ in $\widetilde{\A}_{A}$ is
given by $\{[\varpi]^{k/r}W(\mathcal{O}_{\C}^{\flat})_{A}+p^{k}W(\C^{\flat})_{A}\}_{k\geq1}$.
It suffices to show that
\[
[\varpi]^{2k/r}\A_{\inf,A}\left\langle \frac{p}{[\varpi]^{1/r}}\right\rangle \subset[\varpi]^{k/r}\A_{\inf,A}+p^{k}\A_{\inf,A}\left\langle \frac{p}{[\varpi]^{1/r}}\right\rangle 
\]
inside $\A_{\inf,A}\left\langle p/[\varpi]^{1/r}\right\rangle $,
because the right hand side maps to $[\varpi]^{k/r}W(\mathcal{O}_{\C}^{\flat})_{A}+p^{k}W(\C^{\flat})_{A}$.
(We are implicitly invoking Lemma 7.6 to make sense of this inclusion).

To show this inclusion, start by observing
\[
[\varpi]^{1/r}\A_{\inf,A}\left[\frac{p}{[\varpi]^{1/r}}\right]\subset[\varpi]^{1/r}\A_{\inf,A}+p\A_{\inf,A}\left[\frac{p}{[\varpi]^{1/r}}\right]\subset[\varpi]^{1/r}\A_{\inf,A}+p\A_{\inf,A}\left\langle \frac{p}{[\varpi]^{1/r}}\right\rangle .
\]
Since the right hand side is $p$-adically complete, we deduce that
\[
[\varpi]^{1/r}\A_{\inf,A}\left\langle \frac{p}{[\varpi]^{1/r}}\right\rangle \subset[\varpi]^{1/r}\A_{\inf,A}+p\A_{\inf,A}\left\langle \frac{p}{[\varpi]^{1/r}}\right\rangle .
\]
Arguing inductively, we have
\[
[\varpi]^{k/r}\A_{\inf,A}\left\langle \frac{p}{[\varpi]^{1/r}}\right\rangle \subset[\varpi]^{k/r}\A_{\inf,A}+p[\varpi]^{k-1/r}\A_{\inf,A}+...+p^{k-1}[\varpi]^{1/r}\A_{\inf,A}+p^{k}\A_{\inf,A}\left\langle \frac{p}{[\varpi]^{1/r}}\right\rangle .
\]
Hence,
\[
[\varpi]^{2k/r}\A_{\inf,A}\left\langle \frac{p}{[\varpi]^{1/r}}\right\rangle \subset[\varpi]^{k/r}\A_{\inf,A}+p^{k}\A_{\inf,A}\left\langle \frac{p}{[\varpi]^{1/r}}\right\rangle ,
\]
as required.
\end{proof}
Thus we have a natural continuous map $\widetilde{\A}_{A}^{(0,r]}\rightarrow\widetilde{\A}_{A}$.
\begin{prop}
i. Let $a\in\Z_{\geq1}$. If $p^{a}A=0$, then the kernel of $\widetilde{\A}_{A}^{(0,r],\circ}\rightarrow\widetilde{\A}_{A}$
is killed by $[\varpi]^{a/r}$.

ii. If $A$ is $p$-torsionfree, the map $\widetilde{\A}_{A}^{(0,r],\circ}\rightarrow\widetilde{\A}_{A}$
is injective. 
\end{prop}

\begin{proof}

\emph{i}. Start with the case that $a=1$, so that $A$ is an $\F_{p}$-vector
space. We may choose an isomorphism $A\cong\bigoplus_{i\in I}\F_{p}$.
The map whose kernel we are considering is given by
\[
(\A_{\inf}\widehat{\otimes}_{\Z_{p}}(\oplus_{i\in I}\F_{p}))[X]/(X[\varpi]^{1/r}-p)\rightarrow[\widehat{\oplus}_{i\in I}\mathcal{O}_{\C}^{\flat}]\left[\frac{1}{\varpi}\right],
\]
or, more simply,
\[
[\widehat{\oplus}_{i\in I}\mathcal{O}_{\C}^{\flat}][X]/(X\varpi^{1/r})\rightarrow[\widehat{\oplus}_{i\in I}\mathcal{O}_{\C}^{\flat}]\left[\frac{1}{\varpi}\right],
\]
which maps $X$ to $\frac{p}{\varpi^{1/r}}=0$. Hence the kernel is
given by $X[\widehat{\oplus}_{i\in I}\mathcal{O}_{\C}^{\flat}][X]/(X\varpi^{1/r})$,
which is $\varpi^{1/r}$-torsion.

In general, we have a commutative diagram with exact rows:
\[
\xymatrix{ & \A_{\inf,pA}\left\langle \frac{p}{[\varpi]^{1/r}}\right\rangle _{[\varpi]}^{\wedge}\ar[r]\ar[d] & \A_{\inf,A}\left\langle \frac{p}{[\varpi]^{1/r}}\right\rangle _{[\varpi]}^{\wedge}\ar[r]\ar[d] & \A_{\inf,A/p}\left\langle \frac{p}{[\varpi]^{1/r}}\right\rangle _{[\varpi]}^{\wedge}\ar[r]\ar[d] & 0\\
0\ar[r] & W_{a-1}(\mathcal{O}_{\C}^{\flat})_{pA}\left[\frac{1}{[\varpi]}\right]\ar[r] & W_{a}(\mathcal{O}_{\C}^{\flat})_{A}\left[\frac{1}{[\varpi]}\right]\ar[r] & (\mathcal{O}_{\C}^{\flat}\otimes A/p)_{[\varpi]}^{\wedge}\left[\frac{1}{[\varpi]}\right]
}
.
\]
Here, the top row is exact because it is given by first tensoring
the exact sequence $0\rightarrow pA\rightarrow A\rightarrow A/p\rightarrow0$
with $\A_{\inf}[X]/(X[\varpi]^{1/r}-p)$ and then $[\varpi]$-completing.
According to Proposition 7.3, the $[\varpi]$-torsion in $(\A_{\inf}\otimes_{\Z_{p}}A/p)\left[p/[\varpi]^{1/r}\right]$
is bounded, so by Lemma 7.2 this latter operation preserves exactness.

The bottom row is exact because it is given by first tensoring the
same exact sequence with the flat $\Z_{p}$-module $W(\mathcal{O}_{\C}^{\flat})$,
then completing $[\varpi]$-adically, and then inverting $[\varpi]$.
The second step is exact: this again follows from Lemma 7.2, since
$\mathcal{O}_{\C}^{\flat}\otimes A/p$ is $[\varpi]$-torsionfree.

With the exactness properties of the diagram established, we may use
the snake lemma, from which \emph{i} follows by induction on $a$.

\emph{ii}. If $A=\Z_{p}$, this map is known to be injective. Indeed,
$\widetilde{\A}^{(0,r],\circ}$ can be defined as a subring of $\widetilde{\A}$.

We shall now reduce to the case. Since $A$ is $p$-torsionfree and
$p$-adically complete, we may write $A\cong[\bigoplus_{i\in I}\Z_{p}]_{p}^{\wedge}$
as a $\Z_{p}$-module. We have:

\[
\begin{aligned}\begin{aligned}\widetilde{\A}_{A}^{(0,r],\circ} & =\varprojlim_{b}(\A_{\inf}\otimes_{\Z_{p}}A)\left[\frac{p}{[\varpi]^{1/r}}\right]/[\varpi]^{b}\\
 & \cong\varprojlim_{b}(\A_{\inf}\otimes_{\Z_{p}}[\bigoplus_{i\in I}\Z_{p}]_{p}^{\wedge})\left[\frac{p}{[\varpi]^{1/r}}\right]/[\varpi]^{b}\\
 & =\varprojlim_{b}(\A_{\inf}\otimes_{\Z_{p}}\bigoplus_{i\in I}\Z_{p})\left[\frac{p}{[\varpi]^{1/r}}\right]/[\varpi]^{b}\\
 & =\varprojlim_{b}\bigoplus_{i\in I}\A_{\inf}\left[\frac{p}{[\varpi]^{1/r}}\right]/[\varpi]^{b}\\
 & \hookrightarrow\varprojlim_{b}\prod_{i\in I}\A_{\inf}\left[\frac{p}{[\varpi]^{1/r}}\right]/[\varpi]^{b}\\
 & =\prod_{i\in I}\varprojlim_{b}\A_{\inf}\left[\frac{p}{[\varpi]^{1/r}}\right]/[\varpi]^{b}=\prod_{i\in I}\widetilde{\A}^{(0,r],\circ}.
\end{aligned}
\end{aligned}
\]

On the other hand,
\[
\begin{aligned}\begin{aligned}\widetilde{\A}_{A} & =\varprojlim_{a}W_{a}(\mathcal{O}_{\C}^{\flat})_{A}[\frac{1}{[\varpi]}]\\
 & =\varprojlim_{a}[\varprojlim_{b}(W_{a}(\mathcal{O}_{\C}^{\flat})\otimes_{\Z_{p}}A)/[\varpi]^{b}][\frac{1}{[\varpi]}]\\
 & \cong\varprojlim_{a}[\varprojlim_{b}(W_{a}(\mathcal{O}_{\C}^{\flat})\otimes_{\Z_{p}}[\bigoplus_{i\in I}\Z_{p}]_{p}^{\wedge})/[\varpi]^{b}][\frac{1}{[\varpi]}]\\
 & =\varprojlim_{a}[\varprojlim_{b}(W_{a}(\mathcal{O}_{\C}^{\flat})\otimes_{\Z_{p}}\bigoplus_{i\in I}\Z_{p})/[\varpi]^{b}][\frac{1}{[\varpi]}]\\
 & =\varprojlim_{a}[\varprojlim_{b}\bigoplus_{i\in I}W_{a}(\mathcal{O}_{\C}^{\flat})/[\varpi]^{b}][\frac{1}{[\varpi]}]\\
 & \hookrightarrow\varprojlim_{a}[\varprojlim_{b}\prod_{i\in I}W_{a}(\mathcal{O}_{\C}^{\flat})/[\varpi]^{b}][\frac{1}{[\varpi]}]\\
 & =\varprojlim_{a}[\prod_{i\in I}\varprojlim_{b}W_{a}(\mathcal{O}_{\C}^{\flat})/[\varpi]^{b}][\frac{1}{[\varpi]}]\\
 & =\varprojlim_{a}[\prod_{i\in I}W_{a}(\mathcal{O}_{\C}^{\flat})][\frac{1}{[\varpi]}]\hookrightarrow\varprojlim_{a}\prod_{i\in I}W_{a}(\mathcal{O}_{\C}^{\flat})[\frac{1}{[\varpi]}]\\
 & =\prod_{i\in I}\varprojlim_{a}W_{a}(\mathcal{O}_{\C}^{\flat})[\frac{1}{[\varpi]}]=\prod_{i\in I}\widetilde{\A}.
\end{aligned}
\end{aligned}
\]
We therefore have a commuative diagram
\[
\xymatrix{\widetilde{\A}_{A}^{(0,r],\circ}\ar[r]\ar[d] & \widetilde{\A}_{A}\ar[d]\\
\prod_{i\in I}\widetilde{\A}^{(0,r],\circ}\ar[r] & \prod_{i\in I}\widetilde{\A}
}
\]
where all the maps are have been shown to be injective, except possibly
the top horizontal map. It follows that it is injective also, concluding
the proof.
\end{proof}
Finally, we have the following result.
\begin{thm}
There exists a natural, continuous map $\widetilde{\A}_{A}^{(0,r]}\rightarrow\widetilde{\A}_{A}$.
It is injective. 
\end{thm}

\begin{proof}
It remains to show this map is injective. Recall, by Proposition 7.3.\emph{i}
that $A[p^{N}]=A[p^{\infty}]$ for some $N\gg0$. We have a commutative
diagram
\[
\xymatrix{ & \widetilde{\A}_{A[p^{N}]}^{(0,r],\circ}\ar[r]\ar[d] & \widetilde{\A}_{A}^{(0,r],\circ}\ar[r]\ar[d] & \widetilde{\A}_{A/A[p^{N}]}^{(0,r],\circ}\ar[r]\ar[d] & 0\\
0\ar[r] & \widetilde{\A}_{A[p^{N}]}\ar[r] & \widetilde{\A}_{A}\ar[r] & \widetilde{\A}_{A/A[p^{N}]}
}
.
\]
The top row is exact, because it is obtained by tensoring the sequence
$0\rightarrow A[p^{N}]\rightarrow A\rightarrow A/A[p^{N}]\rightarrow0$
with $\A_{\inf}[p/[\varpi]^{1/r}]$ and then taking $[\varpi]$-adic
completion. This last step is exact because of Lemma 7.2 , since $\A_{\inf}[p/[\varpi]^{1/r}]\otimes A/A[p^{N}]$
has bounded $[\varpi]$-torsion according to Proposition 7.3. The
bottom row is also exact. To see this, start from the exact sequence
$0\rightarrow A[p^{N}]\rightarrow A\rightarrow A/A[p^{N}]\rightarrow0$,
tensor it with $W_{a}(\mathcal{O}_{\C}^{\flat})$, take $[\varpi]$-adic
completion, invert $[\varpi]$, and take inverse limits over $a$.
Here the first step is exact because $A/A[p^{N}]$ is $p$-torsionfree,
and the second step is exact because $W_{a}(\mathcal{O}_{F})\otimes_{\Z_{p}}A/A[p^{N}]$
is $[\varpi]$-torsionfree by Proposition 7.3.

With the exactness established, the snake lemma applies. Using the
previous lemma, we learn that the kernel of $\widetilde{\A}_{A}^{(0,r],\circ}\rightarrow\widetilde{\A}_{A}$
is $[\varpi]$-torsion (even bounded). Since the map $\widetilde{\A}_{A}^{(0,r]}\rightarrow\widetilde{\A}_{A}$
is induced from inverting $[\varpi]$, the proof is finished.
\end{proof}
\begin{cor}
If $s>r$, the natural map $\widetilde{\A}_{A}^{(0,s]}\rightarrow\widetilde{\A}_{A}^{(0,r]}$
is injective. 
\end{cor}

This proves that in the definition of $\widetilde{\A}_{A}^{\dagger}$,
the colimit is in fact a union, so that $\widetilde{\A}_{A}^{\dagger}=\bigcup_{r>0}\widetilde{\A}_{A}^{(0,r]}$,
and $U\subset\widetilde{\A}_{A}^{\dagger}$ is open if and only if
$U\cap\widetilde{\A}_{A}^{(0,r]}$ is open for every $r$. By the
theorem, there is a natural continuous and injective map $\widetilde{\A}_{A}^{\dagger}\rightarrow\widetilde{\A}_{A}$.
By taking $H_{K}$-invariants, we immediately deduce that we have
a similar statement $\widetilde{\A}_{K,A}^{\dagger}=\bigcup_{r>0}\widetilde{\A}_{K,A}^{(0,r]}$
relative to $K$, with a natural continuous and injective map $\widetilde{\A}_{K,A}^{\dagger}\rightarrow\widetilde{\A}_{K,A}$.
\begin{rem}
The analogue of Theorem 7.9 for $\mathcal{\widetilde{\A}}_{A}^{(0,\infty)}$
is also true. Let us explain briefly how this works. The map $\widetilde{\A}_{A}^{(0,\infty),\circ}\rightarrow\widetilde{\A}_{A}$
is injective according to \cite[Rem. 2.2.13]{EG19}. It therefore
suffices to explain why $\widetilde{\A}_{A}^{(0,\infty),\circ}$ is
$[\varpi]$-torsionfree. For each $N\geq1$ we have exact sequences
\[
0\rightarrow p^{N}W(\mathcal{O}_{\C}^{\flat})_{A}\rightarrow p^{N+1}W(\mathcal{O}_{\C}^{\flat})_{A}\rightarrow(\mathcal{O}_{\C}^{\flat})_{p^{N}A/p^{N+1}A}\rightarrow0.
\]
Assume for a moment that $(\mathcal{O}_{\C}^{\flat})_{p^{N}A/p^{N+1}A}$
is $\varpi$-torsionfree. Then it follows by devissage that the $[\varpi]$-torsion
of $\widetilde{\A}_{A}^{(0,\infty),\circ}$ is contained in $\bigcap_{n\geq1}p^{n}\widetilde{\A}_{A}^{(0,\infty),\circ}=0$. 

Renaming $p^{N}A/p^{N+1}A$ as $A$, we reduce to proving that $(\mathcal{O}_{\C}^{\flat})_{A}$
is $\varpi$-torsionfree. Clearly $\mathcal{O}_{\C}^{\flat}$ itself
is $\varpi$-torsionfree, so we have an exact sequence
\[
0\rightarrow\mathcal{O}_{\C}^{\flat}\xrightarrow{\varpi}\mathcal{O}_{\C}^{\flat}\rightarrow\mathcal{O}_{\C}^{\flat}/\varpi\rightarrow0.
\]
Applying $\otimes_{\Z_{p}}A$ to $\mathcal{O}_{\C}^{\flat}$ is the
same as applying $\otimes_{\F_{p}}A/p$ to it. Since $A/p$ is free,
tensoring with it gives an exact sequence
\[
0\rightarrow\mathcal{O}_{\C}^{\flat}\otimes A\xrightarrow{\varpi}\mathcal{O}_{\C}^{\flat}\otimes A\rightarrow\mathcal{O}_{\C}^{\flat}/\varpi\otimes A\rightarrow0.
\]
Now $\mathcal{O}_{\C}^{\flat}/\varpi\otimes A$ is killed by $\varpi$,
and in particular, its $\varpi$-torsion is bounded. Hence, by Lemma
7.2, the sequence stays exact after $\varpi$-adic completion:
\[
0\rightarrow(\mathcal{O}_{\C}^{\flat})_{A}\xrightarrow{\varpi}(\mathcal{O}_{\C}^{\flat})_{A}\rightarrow(\mathcal{O}_{\C}^{\flat})_{A}/\varpi\rightarrow0.
\]
It follows that $(\mathcal{O}_{\C}^{\flat})_{A}$ is $\varpi$-torsionfree,
as required.
\end{rem}

We shall now deduce similar properties for the imperfect rings relative
to $K$.
\begin{prop}
The natural map $\A_{K,A}\rightarrow\widetilde{\A}_{A}$ is injective.
\end{prop}

\begin{proof}
This follows from \cite[Rem. 2.2.13]{EG21}.
\end{proof}
\begin{prop}
The natural map $\A_{K,A}^{(0,r],\circ}\rightarrow\widetilde{\A}_{A}^{(0,r],\circ}$
is injective.
\end{prop}

\begin{proof}
We do this in several steps.

\textbf{Step 1. The statement is true if $A$ is $p$-torsionfree.}

Since $\A_{K}^{(0,r],\circ}=\widetilde{\A}^{(0,r],\circ}\cap\A_{K}$,
the intersection taken in $\widetilde{\A}$, we get an exact sequence
\[
0\rightarrow\A_{K}^{(0,r],\circ}\rightarrow\widetilde{\A}^{(0,r],\circ}\oplus\A_{K}\xrightarrow{(x,y)\mapsto x-y}\widetilde{\A}.
\]
Tensoring with $A$ we get
\[
0\rightarrow\A_{K}^{(0,r],\circ}\otimes A\rightarrow(\widetilde{\A}^{(0,r],\circ}\otimes A)\oplus(\A_{K}\otimes A)\xrightarrow{(x,y)\mapsto x-y}\widetilde{\A}\otimes A,
\]
so that $\A_{K}^{(0,r],\circ}\otimes A=(\widetilde{\A}^{(0,r],\circ}\otimes A)\cap(\A_{K}\otimes A)$,
the intersection taken in $\widetilde{\A}\otimes A$.

Next, we notice that $(\widetilde{\A}\otimes A)/(\A_{K}\otimes A)$
is $T$-torsionfree. Indeed, $T$ is invertible in $\A_{K}$ and $\widetilde{\A}$.
Hence, 
\[
(\widetilde{\A}^{(0,r],\circ}\otimes A)/(\A_{K}^{(0,r],\circ}\otimes A)=(\widetilde{\A}^{(0,r],\circ}\otimes A)/(\widetilde{\A}^{(0,r],\circ}\otimes A)\cap(\A_{K}\otimes A),
\]
which injects into $(\widetilde{\A}\otimes A)/(\A_{K}\otimes A)$,
is also $T$-torsionfree. 

Now consider the commutative diagram
\[
\xymatrix{0\ar[r] & \A_{K}^{(0,r],\circ}\otimes A\ar[r]\ar[d]^{T^{a}} & \widetilde{\A}^{(0,r],\circ}\otimes A\ar[r]\ar[d]^{T^{a}} & (\widetilde{\A}^{(0,r],\circ}\otimes A)/(\A_{K}^{(0,r],\circ}\otimes A)\ar[r]\ar[d]^{T^{a}} & 0\\
0\ar[r] & \A_{K}^{(0,r],\circ}\otimes A\ar[r] & \widetilde{\A}^{(0,r],\circ}\otimes A\ar[r] & (\widetilde{\A}^{(0,r],\circ}\otimes A)/(\A_{K}^{(0,r],\circ}\otimes A)\ar[r] & 0
}
.
\]
By the snake lemma, the maps $(\A_{K}^{(0,r],\circ}\otimes A)/T^{a}\rightarrow(\widetilde{\A}^{(0,r],\circ}\otimes A)/T^{a}$
are injective. Taking the inverse limit $a\rightarrow\infty$, this
implies $\A_{K,A}^{(0,r],\circ}\rightarrow\widetilde{\A}_{A}^{(0,r],\circ}$
is injective.

\textbf{Step 2. The statement is true if $pA=0$.}

First we claim that the quotient $\widetilde{\A}/\A_{K}$ is $p$-torsionfree.
Indeed, this follows from the fact that $\widetilde{\A}$ is $p$-torsionfree
and the injectivity of $\A_{K}\otimes\F_{p}\rightarrow\widetilde{\A}\otimes\F_{p}$.
Now since $\widetilde{\A}^{(0,r],\circ}/\A_{K}^{(0,r],\circ}=\widetilde{\A}^{(0,r],\circ}/\widetilde{\A}^{(0,r],\circ}\cap\A_{K}$
injects into $\widetilde{\A}/\A_{K}$, it is also $p$-torsionfree.
In particular, the natural map from $\A_{K}^{(0,r],\circ}\otimes\F_{p}$
into $\widetilde{\A}^{(0,r],\circ}\otimes\F_{p}$ is injective. 

Next, upon choosing an isomorphism $A\cong\bigoplus_{i\in I}\F_{p}$,
we have a commutative diagram
\[
\xymatrix{\bigoplus_{i\in I}(\A_{K}^{(0,r],\circ}\otimes\F_{p})/T^{a}\ar[r]\ar[d] & \bigoplus_{i\in I}(\widetilde{\A}^{(0,r],\circ}\otimes\F_{p})/T^{a}\ar[d]\\
\prod_{i\in I}(\A_{K}^{(0,r],\circ}\otimes\F_{p})/T^{a}\ar[r] & \prod_{i\in I}(\widetilde{\A}^{(0,r],\circ}\otimes\F_{p})/T^{a}
}
.
\]
Here, the vertical maps are injective. Taking the limit over $a$,
we obtain
\[
\xymatrix{\A_{K,A}^{(0,r],\circ}\ar[r]\ar[d] & \widetilde{\A}_{A}^{(0,r],\circ}\ar[d]\\
\prod_{i\in I}\A_{K}^{(0,r],\circ}\otimes\F_{p}\ar[r] & \prod_{i\in I}\widetilde{\A}^{(0,r],\circ}\otimes\F_{p}
}
,
\]
where the vertical maps are still injective and the lower horizontal
map is injective by we have just explained. It follows that the top
horizontal map is injective.

\textbf{Step 3. If $p^{N}A=0$ for some $N$, the statement is true
for $A$.}

To case $N=1$ was already discussed. Now consider the commutative
diagram
\[
\xymatrix{ & \A_{K,pA}^{(0,r],\circ}\ar[r]\ar[d] & \A_{K,A}^{(0,r],\circ}\ar[r]\ar[d] & \A_{K,A/p}^{(0,r],\circ}\ar[r]\ar[d] & 0\\
0\ar[r] & \widetilde{\A}_{pA}^{(0,r],\circ}\ar[r] & \widetilde{\A}_{A}^{(0,r],\circ}\ar[r] & \widetilde{\A}_{A/p}^{(0,r],\circ}\ar[r] & 0
}
.
\]
The top row is exact because it is obtained from tensoring $0\rightarrow pA\rightarrow A\rightarrow A/p\rightarrow0$
with $\A_{K}^{(0,r],\circ}$, which is $p$-torsionfree, and then
completing with respect to $T$, which is right exact by Nakayama's
lemma. The bottom row is exact because of Lemma 7.2, since $\widetilde{\A}^{(0,r],\circ}\otimes A/p$
has bounded $[\varpi]$-torsion (Proposition 7.3.$iv$).

The result now easily follows by devissage using the snake lemma.

\textbf{Step 4. The statement is true for general $A$.}

This is similar to the proof of step 3. Namely, take $N$ large enough
so that $A[p^{N}]=A[p^{\infty}]$. Then one has a commutative diagram
\[
\xymatrix{ & \A_{K,A[p^{N}]}^{(0,r],\circ}\ar[r]\ar[d] & \A_{K,A}^{(0,r],\circ}\ar[r]\ar[d] & \A_{K,A/A[p^{N}]}^{(0,r],\circ}\ar[r]\ar[d] & 0\\
0\ar[r] & \widetilde{\A}_{A[p^{N}]}^{(0,r],\circ}\ar[r] & \widetilde{\A}_{A}^{(0,r],\circ}\ar[r] & \widetilde{\A}_{A/A[p^{N}]}^{(0,r],\circ}\ar[r] & 0
}
,
\]
which is exact by the same arguments. We then conclude by using what
was prove in steps 1 and 3.
\end{proof}
\begin{cor}
The natural map $\A_{K,A}^{(0,r]}\rightarrow\A_{K,A}$ is injective.
\end{cor}

\begin{proof}
From the proposition it follows that $\A_{K,A}^{(0,r]}\rightarrow\widetilde{\A}_{A}^{(0,r]}$
is injective. Use the diagram
\[
\xymatrix{\A_{K,A}^{(0,r]}\ar[r]\ar[d] & \A_{K,A}\ar[d]\\
\widetilde{\A}_{A}^{(0,r]}\ar[r] & \widetilde{\A}_{A}
}
,
\]
in which we already know by Theorem 7.9, Proposition 7.12 and Proposition
7.13 that every map other than $\A_{K,A}^{(0,r]}\rightarrow\A_{K,A}$
is injective.
\end{proof}
\begin{cor}
If $s>r$, the natural map $\A_{K,A}^{(0,s]}\rightarrow\A_{K,A}^{(0,r]}$
is injective.
\end{cor}

As in the perfect case, we now know that $\A_{K,A}^{\dagger}=\bigcup_{r>0}\A_{K,A}^{(0,r]}$.

\subsection{Compatibility with reduction}

In this subsection we establish a result that will be used in $\mathsection3$.

For $a\geq1$ we have natural maps $\widetilde{\A}_{K,A/p^{a+1}}\rightarrow\widetilde{\A}_{K,A/p^{a}}$,
obtain from the surjection $A/p^{a+1}\rightarrow A/p^{a}$ by tensoring
with $W(\mathcal{O}_{\widehat{K}_{\infty}}^{\flat})$, completing
$[\varpi]$-adically, and inverting $[\varpi]$.
\begin{lem}
Suppose $A$ is $p$-torsionfree. Then for $N\in\Z_{\geq1}$, the
natural map 
\[
p^{N}\widetilde{\A}_{K,A}\rightarrow\varprojlim_{a}p^{N}\widetilde{\A}_{K,A/p^{a}}
\]
 is an isomorphism.
\end{lem}

\begin{proof}
As $A$ is $p$-torsionfree, for $a\geq N$ we have exact sequences
\[
0\rightarrow p^{a-N}A/p^{a}A\rightarrow A/p^{a}\xrightarrow{p^{N}}A/p^{a}\rightarrow A/p^{N}\rightarrow0.
\]
Hence, tensoring with $W(\mathcal{O}_{\widehat{K}_{\infty}}^{\flat})$,
for $b\geq a$ we have
\[
0\rightarrow W_{b}(\mathcal{O}_{\widehat{K}_{\infty}}^{\flat})\otimes p^{a-N}A/p^{a}A\rightarrow W_{b}(\mathcal{O}_{\widehat{K}_{\infty}}^{\flat})\otimes A/p^{a}\xrightarrow{p^{N}}W_{b}(\mathcal{O}_{\widehat{K}_{\infty}}^{\flat})\otimes A/p^{a}\rightarrow W_{b}(\mathcal{O}_{\widehat{K}_{\infty}}^{\flat})\otimes A/p^{N}\rightarrow0.
\]
Each term appearing in the exact sequence has no $[\varpi]$-torsion,
by Proposition 7.3. Hence, by Lemma 7.2, completing $[\varpi]$-adically
is exact. Inverting $[\varpi]$ also, we get an exact sequence 
\[
0\rightarrow\widetilde{\A}_{K,p^{a-N}A/p^{a}}\rightarrow\widetilde{\A}_{K,A/p^{a}}\xrightarrow{p^{N}}p^{N}\widetilde{\A}_{K,A/p^{a}}\rightarrow0.
\]
Taking limits, we have a long exact sequence
\[
0\rightarrow\varprojlim_{a}\widetilde{\A}_{K,p^{a-N}A/p^{a}}\rightarrow\widetilde{\A}_{K,A}=\varprojlim_{a}\widetilde{\A}_{K,A/p^{a}}\xrightarrow{p^{N}}\varprojlim_{a}p^{N}\widetilde{\A}_{K,A/p^{a}}\rightarrow\R^{1}\varprojlim_{a}\widetilde{\A}_{K,p^{a-N}A/p^{a}},
\]
but $\varprojlim_{a}\widetilde{\A}_{K,p^{a-N}A/p^{a}}$ and $\R^{1}\varprojlim_{a}\widetilde{\A}_{K,p^{a-N}A/p^{a}}$
both vanish because the composition of $N$ successive transition
maps $\widetilde{\A}_{K,p^{a+1-N}A/p^{a+1}}\rightarrow\widetilde{\A}_{K,p^{a-N}A/p^{a}}$
is zero. This proves the lemma.
\end{proof}
\begin{prop}
For $N\in\Z_{\geq1}$ and $1/r\in\Z[1/p]_{>0}$ we have natural isomorphisms
\[
\widetilde{\A}_{K,A}/p^{N}\widetilde{\A}_{K,A}\cong\widetilde{\A}_{K,A/p^{N}}\cong\widetilde{\A}_{K,A/p^{N}}^{(0,r]}\cong\widetilde{\A}_{K,A/p^{N}}^{(0,\infty)}.
\]
\end{prop}

\begin{proof}
\textbf{}There are a priori natural inclusions $\widetilde{\A}_{K,A/p^{N}}^{(0,\infty)}\subset\widetilde{\A}_{K,A/p^{N}}^{(0,r]}\subset\widetilde{\A}_{K,A/p^{N}}$
according to Theorem 7.9, Corollary 7.10 and Remark 7.11. The inclusion
of the first term in the last term is an isomorphism by Proposition
7.1.$ii$. Hence the middle term is also naturally isomorphic to them.

It remains to construct $\widetilde{\A}_{K,A}/p^{N}\widetilde{\A}_{K,A}\cong\widetilde{\A}_{K,A/p^{N}}$.
Starting from 
\[
0\rightarrow p^{N}(A/p^{a})\rightarrow A/p^{a}\rightarrow A/p^{N}\rightarrow0,
\]
for $a\geq N$, tensor with $W(\mathcal{O}_{\widehat{K}_{\infty}}^{\flat})$,
complete $[\varpi]$-adically, and invert $[\varpi]$ to get
\[
0\rightarrow\widetilde{\A}_{K,p^{N}(A/p^{a})}\rightarrow\widetilde{\A}_{K,A/p^{a}}\rightarrow\widetilde{\A}_{K,A/p^{N}}\rightarrow0.
\]
Taking limits, we obtain a surjective map
\[
\widetilde{\A}_{K,A}\cong\varprojlim_{a}\widetilde{\A}_{K,A/p^{a}}\rightarrow\widetilde{\A}_{K,A/p^{N}}
\]
which factors through $\widetilde{\A}_{K,A}/p^{N}\widetilde{\A}_{K,A}$. 

\textbf{Step 1. }The map $\widetilde{\A}_{K,A}\rightarrow\widetilde{\A}_{K,A/p^{N}}$
is an isomorphism if $A$ is killed by a power of $p$.

In this case, consider the exact sequence
\[
0\rightarrow A[p^{N}]\rightarrow A\xrightarrow{p^{N}}A\rightarrow A/p^{N}\rightarrow0.
\]
Now tensor with $W(\mathcal{O}_{\widehat{K}_{\infty}}^{\flat})$,
complete $[\varpi]$-adically (which is exact by Lemma 7.2 and Proposition
7.3), and invert $[\varpi]$ to get
\[
0\rightarrow\widetilde{\A}_{K,A[p^{N}]}\rightarrow\widetilde{\A}_{K,A}\xrightarrow{p^{N}}\widetilde{\A}_{K,A}\rightarrow\widetilde{\A}_{K,A/p^{N}}\rightarrow0
\]
(there is no need to take an inverse limit since all the terms will
be the same for $a\gg0$). In particular, $\widetilde{\A}_{K,A}/p^{N}\widetilde{\A}_{K,A}\cong\widetilde{\A}_{K,A/p^{N}}$.

\textbf{Step 2. }The map $\widetilde{\A}_{K,A}/p^{N}\rightarrow\widetilde{\A}_{K,A/p^{N}}$
is an isomorphism if $A$ is $p$-torsionfree.

In this case, consider the exact sequence
\[
0\rightarrow p^{a-N}A/p^{a}A\rightarrow A/p^{a}\xrightarrow{p^{N}}A/p^{a}\rightarrow A/p^{N}\rightarrow0.
\]
Now tensor with $W(\mathcal{O}_{\widehat{K}_{\infty}}^{\flat})$,
complete $[\varpi]$-adically (which is exact by Lemma 7.2 and Proposition
7.3), and invert $[\varpi]$ to get
\[
0\rightarrow\widetilde{\A}_{K,p^{a-N}A/p^{a}A}\rightarrow\widetilde{\A}_{K,A/p^{a}}\xrightarrow{p^{N}}\widetilde{\A}_{K,A/p^{a}}\rightarrow\widetilde{\A}_{K,A/p^{N}}\rightarrow0.
\]
In particular, we have a short exact sequence
\[
0\rightarrow p^{N}\widetilde{\A}_{K,A/p^{a}}\rightarrow\widetilde{\A}_{K,A/p^{a}}\rightarrow\widetilde{\A}_{K,A/p^{N}}\rightarrow0.
\]
Taking limits over $a$, we have
\[
0\rightarrow\varprojlim_{a}p^{N}\widetilde{\A}_{K,A/p^{a}}\rightarrow\widetilde{\A}_{K,A}=\varprojlim\widetilde{\A}_{K,A/p^{a}}\rightarrow\widetilde{\A}_{K,A/p^{N}}.
\]
Hence, the kernel of $\widetilde{\A}_{K,A}\rightarrow\widetilde{\A}_{K,A/p^{N}}$
is equal to $\varprojlim_{a}p^{N}\widetilde{\A}_{K,A/p^{a}}=p^{N}\widetilde{\A}_{K,A}$,
by the previous lemma.

\textbf{Step 3.} The map is an isomorphism in general.

Indeed, let $A[p^{\infty}]=A[p^{M}]$ for $M\gg0$ so that $A/A[p^{M}]$
is $p$-torsionfree and $A[p^{M}]$ is killed by $p^{M}$. We have
a commutative diagram:
\[
\xymatrix{0\ar[r] & \widetilde{\A}_{K,A[p^{M}]}/p^{N}\ar[d]\ar[r] & \widetilde{\A}_{K,A}/p^{N}\ar[d]\ar[r] & \widetilde{\A}_{K,A/A[p^{M}]}/p^{N}\ar[d]\ar[r] & 0\\
0\ar[r] & \widetilde{\A}_{K,A[p^{M}]/p^{N}}\ar[r] & \widetilde{\A}_{K,A/p^{N}}\ar[r] & \widetilde{\A}_{K,A/(A[p^{M}]+p^{N}A)}\ar[r] & 0
}
.
\]
The top row is exact because it is obtained by starting from 
\[
0\rightarrow A[p^{M}]\rightarrow A/p^{a}A\rightarrow A/(A[p^{M}]+p^{a}A)\rightarrow0
\]
 for $a\geq N$, tensoring with $W(\mathcal{O}_{\widehat{K}_{\infty}}^{\flat})$,
completing $[\varpi]$-adically, inverting $[\varpi]$, taking inverse
limits over $a$, and tensoring with $\Z/p^{N}$. All of these operations
are exact, including the last two ones: inverse limits over $a$ because
the transition maps $\widetilde{\A}_{K,A[p^{N}]/p^{a+1}}\rightarrow\widetilde{\A}_{K,A[p^{N}]/p^{a}}$
are just the identity for $a\geq N$, and tensoring over $\F_{p}$
because $\widetilde{\A}_{K,A/A[p^{M}]}$ is $p$-torsionfree (as can
be seen from the proof of Proposition 7.8.$ii$). The bottom row is
exact, because it is obtained by starting from 
\[
0\rightarrow A[p^{M}]\rightarrow A\rightarrow A/A[p^{M}]\rightarrow0,
\]
tensoring with $W_{N}(\mathcal{O}_{\widehat{K}_{\infty}}^{\flat})$,
completing $[\varpi]$-adically and inverting $[\varpi]$-adically.
The second step is exact because $A/A[p^{M}]$ is $p$-torsionfree.

With this given, one now concludes the proof from the snake lemma,
using steps 1 and 2.
\end{proof}

\subsection{The $(\varphi,G_{K})$-actions}

There are continuous $(\varphi,G_{K})$-actions on $\widetilde{\A}_{A}$
by \cite[Lem. 2.2.18]{EG19}. We now establish the continuity on all
of the rings defined in $\mathsection2.1$.

Given $r\in\bbR_{>0}\cup\{\infty\}$ we have a continuous map 
\[
\varphi:\widetilde{\A}_{A}^{(0,r]}\rightarrow\widetilde{\A}_{A}^{(0,r/p]}
\]
induced by extending the action of $\varphi$ on $\A_{\inf}$. It
is continuous since $\varphi([\varpi])=[\varpi]^{p}$ and the topology
is $[\varpi]$-adic on both the source and the target.

Similarly, we have a continuous inverse
\[
\varphi^{-1}:\widetilde{\A}_{A}^{(0,r/p]}\rightarrow\widetilde{\A}_{A}^{(0,r]}.
\]

This immediately extends to give continuous $\varphi$ and $\varphi^{-1}$
actions on $\widetilde{\A}_{A}^{\dagger}$.
\begin{lem}
Let $G$ be a topological group.

i. If $G$ is profinite and acts continuously on topological spaces
$X_{1}\rightarrow X_{2}\rightarrow...$., and $X=\varinjlim_{i}X_{i}$
is endowed with the direct limit topology, then the natural action
of $G$ on $X$ is continuous.

ii. If $G$ acts continuously on topological spaces $X_{1}\leftarrow X_{2}\leftarrow X_{3},...$
then the natural action of $G$ on the limit $\varprojlim_{i}X_{i}$
is continuous.
\end{lem}

\begin{proof}
$i.$ Let $U\subset X$ and let $(g,x)\in G\times X$ be such that
$\mathrm{act}:G\times X\rightarrow X$ maps $(g,x)$ into $U$. Suppose
$x\in X_{n}$, and let $i_{n}:X_{n}\rightarrow X$ denote the canonical
map. Then since $\mathrm{act}|_{G\times X_{n}}$ is continuous, there
exists $N\subset G$ open compact such that 
\[
\mathrm{act}|_{G\times X_{n}}(Ng\times\{x\})\subset i_{n}^{-1}(U),
\]
which implies 
\[
\mathrm{act}|_{G\times X}(Ng\times\{x\})\subset U.
\]
By the tube lemma, there exists an open subset $V\subset X$ containing
$x$ such that $\mathrm{act}(Ng\times V)\subset U$. This proves that
the action is continuous.

$ii.$ It is enough to prove the continuity of the action on the product
$\prod X_{i}$, which is obvious. 
\end{proof}
Now, $G_{K}$ acts continuously on $\A_{\inf}$, hence on $(\A_{\inf}\otimes_{\Z_{p}}A)[p/[\varpi]^{1/r}]/[\varpi]^{a}$.
As $\widetilde{\A}_{A}^{(0,r]}$ is built out of $(\A_{\inf}\otimes_{\Z_{p}}A)[p/[\varpi]^{1/r}]/[\varpi]^{a}$
by taking direct limits and projective limits, the lemma implies that
the action of $G_{K}$ on $\widetilde{\A}_{A}^{(0,r]}$ is continuous.
Finally, $\widetilde{\A}_{A}^{\dagger}$ is built out of $\widetilde{\A}_{A}^{(0,r]}$
by taking direct limits, so again by the lemma the action of $G_{K}$
on it is continuous. Via the topological embeddings $\A_{K,A}^{\dagger}\hookrightarrow\widetilde{\A}_{K,A}^{\dagger}\hookrightarrow\widetilde{\A}_{A}^{\dagger}$
and $\A_{K,A}\hookrightarrow\widetilde{\A}_{K,A}\hookrightarrow\widetilde{\A}_{A}$,
we conclude this subsection with the following result.
\begin{prop}
The $(\varphi^{\pm1},G_{K})$-actions (resp. $(\varphi^{\pm1},\Gamma_{K})$-actions,
resp. $(\varphi,\Gamma_{K})$-actions) on $\widetilde{\A}_{A}^{\dagger}$
and $\widetilde{\A}_{A}$ (resp. $\widetilde{\A}_{K,A}^{\dagger}$
and $\widetilde{\A}_{K,A}$, resp. $\A_{K,A}^{\dagger}$ and $\A_{K,A}$)
are continuous. 
\end{prop}

\end{document}